\documentclass{amsart}
\usepackage{amsmath,amssymb,amsfonts}
\usepackage{graphicx,xcolor}
\usepackage{epstopdf}
\usepackage[colorlinks=true]{hyperref}
\hypersetup{urlcolor=blue, citecolor=red}
\usepackage{bbm}

\def\quotient#1#2{%
    \lower 0.2ex\hbox{$#1$}\big\backslash \raise0.2ex\hbox{$#2$}%
    }
    
\newcommand{\hl}{\sl}
\newcommand{\hd}{\sl}
\newcommand{\ba}{\begin{array}}
\newcommand{\ea}{\end{array}}

\newcommand{\LL}{\mbox{\rm L}}
\newcommand{\Cnt}{\mbox{\rm C}}
\newcommand{\id}{\mbox{\rm id}}

\renewcommand{\d}{\mathrm{d}}
\newcommand{\e}{\mathrm{e}}

\newcommand{\SX}{{\mathcal G}}
\newcommand{\reg}{{\mathcal R}}
\newcommand{\zero}{{\mathcal Z}}
\newcommand{\jac}{{\mathcal J}}

\newcommand{\bs}{{\mathcal B}}
\newcommand{\Gr}{{Gr}}

\newcommand{\NN}{\mathbb{N}} \newcommand{\ZZ}{\mathbb{Z}}
 \newcommand{\RR}{\mathbb{R}}

\newcommand{\XX}{X}
\newcommand{\ganz}{\overline{\XX}}
\newcommand{\rand}{\partial\XX}
\newcommand{\chr}{{\mathbbm 1}}
\newcommand{\Lim}{L_\Gamma}
\newcommand{\radlim}{L_\Gamma^{\small{\mathrm{rad}}}}

\newcommand{\at}{\!\cdot\!}
\newcommand{\width}{\mathrm{width}}
\newcommand{\ndpt}{\partial}
\newcommand{\Hopf}{H}
\newcommand{\xo}{{o}}
\newcommand{\gamo}{{\gamma\xo}}

\newcommand{\be}{\begin{eqnarray*}}
\newcommand{\ee}{\end{eqnarray*}}
\newcommand{\ein}{\,\rule[-5pt]{0.4pt}{12pt}\,{}}

\newcommand{\is}{\mbox{Is}}

\newcommand{\supp}{\mbox{supp}}

\newcommand{\st}{\mbox{such}\ \mbox{that}\ }
\newcommand{\wrt}{\mbox{with}\ \mbox{respect}\ \mbox{to}\ }

\newtheorem{theorem}{Theorem}[section]
\newtheorem{corollary}{Corollary}
\newtheorem*{mainA}{Theorem A}
\newtheorem*{mainB}{Theorem B}
\newtheorem*{mainC}{Theorem C}
\newtheorem{lemma}[theorem]{Lemma}
\newtheorem{proposition}{Proposition}

\theoremstyle{definition}
\newtheorem{definition}[theorem]{Definition}

\title[Hopf-Tsuji-Sullivan dichotomy for quotients of Hadamard spaces] 
      {Hopf-Tsuji-Sullivan dichotomy for quotients of \\ Hadamard spaces with a rank one isometry}

\author[Gabriele Link]{}

\subjclass{Primary: 37D40, 20F67; Secondary: 37D25.}
 \keywords{rank one space, Hopf-Tsuji-Sullivan dichotomy, geodesic currents, Patterson-Sullivan measures.}

 \email{gabriele.link@kit.edu}
 
\begin{document}
\bibliographystyle{AIMS.bst}
\maketitle

\centerline{\scshape Gabriele Link$^*$}
\medskip
{\footnotesize
 \centerline{Institut f\"ur Algebra und Geometrie}
   \centerline{Karlsruhe Institute of Technology (KIT)}
   \centerline{Englerstr.~2, 76 131 Karlsruhe, Germany}
} 

\bigskip


\begin{abstract}
Let $\XX$ be a proper Hadamard space  and $\ \Gamma<\is(\XX)$ a non-elementary discrete group of isometries with a rank one isometry. We discuss and prove Hopf-Tsuji-Sullivan dichotomy for the geodesic flow on the set of parametrized geodesics of the quotient $\quotient{\Gamma}{\XX}$ and with respect to Ricks' measure introduced in \cite{Ricks}. This generalizes previous work of the author and J.~C.~Picaud on Hopf-Tsuji-Sullivan dichotomy in the analogous manifold setting and with respect to Knieper's measure. 
\end{abstract}

\section{Introduction}
Let $(\XX,d)$ be a  proper  Hadamard space and $\Gamma<\is(\XX)$ a discrete group.  Let $\SX$ denote the  set of parametrized geodesic lines in $\XX$ endowed with the compact-open topology (which can be identified with the unit tangent  bundle $S\XX$ if $\XX$ is a Riemannian manifold) and consider the action of $\RR$ on $\SX$ by reparametrization. This action induces a flow $g_\Gamma$ on the quotient space $\quotient{\Gamma}{\SX}$.  Let $m_\Gamma$ be an appropriate Radon  measure on $\quotient{\Gamma}{\SX}$ which is invariant by the  flow $g_\Gamma$. Hopf-Tsuji-Sullivan dichotomy then states that -- under certain conditions on the space $\XX$ and the group $\Gamma$ -- there are precisely two mutually exclusive possibilities for the dynamical system $( \quotient{\Gamma}{\SX}, g_\Gamma, m_\Gamma)$: Either it is  conservative (that is almost every orbit is recurrent)  and ergodic (which means that the only invariant sets have zero or full measure) or it is dissipative (that is almost every orbit is  divergent) and non-ergodic. For a precise definition of the previous notions the reader is referred to Section~\ref{dyndef}.

The story of Hopf-Tsuji-Sullivan dichotomy probably began with Poincar\'e's
 recurrence theorem applied to Riemann surfaces and with Hopf's seminal
 work  later in the 1930's (see \cite{Hopf} and \cite{MR0284564}). 
For quotients of the hyperbolic plane by Fuchsian groups it was  
observed  that  with respect to Liouville measure  the geodesic flow  is either conservative and ergodic   or dissipative and non-ergodic. 
Later, with the invention of the remarkable Patterson-Sullivan measures on the boundary of $\XX$ (see \cite{MR0450547} and  \cite{MR556586} for the original constructions,
then  \cite{MR1348871}, \cite{MR1293874}, \cite{MR1207579}  for extensions, and \cite{MR2057305} for a clear account and deep applications of this theory) and then the  construction of Bowen-Margulis measure  on $\quotient{\Gamma}{ S \XX}$ using these,  generalizations to a wider class of spaces and groups have been obtained by several authors. Among them I only want to mention here the work of M.~Coornaert and A.~Papadopoulos (\cite{MR1207579}) which deals with locally compact metric trees and the work of V.~Kaimanovich (\cite{MR1293874}) in the setting of Gromov hyperbolic spaces with some additional properties; these were probably the first ones considering non-Riemannian   spaces.  T.~Roblin (\cite[Th{\'e}or{\`e}me 1.7]{MR2057305}) then gave a unified version for all proper CAT$(-1)$-spaces. 

Recently, in \cite{LinkPicaud}, Hopf-Tsuji-Sullivan dichotomy was proved for quotients of Hadamard {\hl manifolds} by discrete isometry groups containing an element which translates a geodesic without parallel perpendicular Jacobi field 
and with respect to Knieper's  measure (\cite{MR1652924}) on the unit tangent bundle.  The main goal of the present paper is to prove Hopf-Tsuji-Sullivan dichotomy in the setting of proper Hadamard {\hl spaces} with a {\hl rank one} isometry (that is an isometry translating a geodesic which does not bound a flat half-plane) and hence to generalize the Main Theorem of \cite{LinkPicaud} to non-Riemannian spaces; compared to \cite{LinkPicaud} we also impose an a priori weaker condition on the discrete group $\Gamma$ of the Hadamard {\hl manifold} $\XX$: 
In fact,  we 
only need a discrete group with infinite limit set which contains the fixed point of a rank one isometry of $\XX$. So in particular $\XX$ need not a priori possess a geodesic without parallel perpendicular Jacobi field, but only one without a flat half-plane. However, this can only happen when  $\XX$ does not admit a quotient of finite volume according to  the rank rigidity  theorem of
Ballmann \cite{MR819559} and Burns-Spatzier \cite{MR908215}, which  asserts that otherwise $\XX$ has a geodesic without parallel perpendicular Jacobi field. 

Even though some of the results from the above mentioned paper \cite{LinkPicaud} remain true in this more general setting, 
there are several obstructions occurring when singular spaces are involved. The probably most important one is the fact that Knieper's measure 
cannot be constructed without a volume form on the closed and convex subsets corresponding to the parallel sets of geodesic lines.  
We will therefore follow here the construction proposed by R.~Ricks in \cite{Ricks} and  first define weak Bowen-Margulis measure on the quotient $\quotient{\Gamma}{[\SX]}$ of parallel classes of {\hl parametrized} geodesic lines by $\Gamma$. With respect to this measure we have the following 
 \begin{mainA} Let $X$ be a proper  Hadamard space and $\ \Gamma<\is(X)$ a  discrete group with the fixed point of a rank one isometry of $X$ in its infinite limit set.
Then with respect to Ricks' weak Bowen-Margulis measure either the geodesic flow on $\quotient{\Gamma}{[\SX]}$ is conservative, or it is dissipative and non-ergodic unless the measure is supported on a single orbit by the geodesic flow on $\quotient{\Gamma}{[\SX]}$. 
\end{mainA}
Notice that since Ricks' construction of weak Bowen-Margulis measure depends on the choice of a conformal density, a priori there may exist many distinct weak Bowen-Margulis measures. In the conservative case however, it is well-known that up to scaling there exists only one  conformal density; hence there is precisely one Ricks'  weak Bowen-Margulis measure  in this setting.

We remark that we do not manage to deduce ergodicity from conservativity in this weakest setting (only requiring the fixed point of an {\hl arbitrary} rank one isometry in the limit set of $\Gamma$) as neither the Hopf argument nor Kaimanovich's method for the proof of Theorem~2.5 in \cite{MR1293874} can be applied in this case. However, if $\XX$ is {\hl geodesically complete} then thanks to Proposition~\ref{largewidthgiveszerowidth} this weak assumption implies the existence of
 a {\hl zero width} rank one geodesic (that is one which does not even bound  a flat {\hl strip}) with extremities  in the limit set of $\Gamma$.  
Under this additional assumption  any weak Bowen-Margulis measure induces a so-called Ricks' Bowen-Margulis measure $m_\Gamma$ on the quotient  $\quotient{\Gamma}{\SX}$. Notice that by the remark following Theorem~A there is only one Ricks' Bowen-Margulis measure in the conservative case. We finally get Theorem~\ref{HTS}, the full Hopf-Tsuji-Sullivan dichotomy including  ergodicity in the conservative case; a short version reads as follows:
\begin{mainB} Let $X$ be a proper  Hadamard space and $\ \Gamma<\is(X)$ a  discrete group with the fixed point of a rank one isometry of $\XX$ {\hl and} the extremities of a {\hl zero width} rank one geodesic  in its infinite limit set. Then with respect to any Ricks' Bowen-Margulis measure either the geodesic flow on $\quotient{\Gamma}{\SX}$ is conservative and ergodic, or it is dissipative and non-ergodic unless the measure is supported on a single orbit by the geodesic flow in $\quotient{\Gamma}{\SX}$. \end{mainB}
We finally want to mention here that if $\XX$  is a Hadamard {\hl manifold}, then in the conservative case Ricks' Bowen-Margulis measure $m_\Gamma$ is equal to  Knieper's  measure which was used  in \cite{LinkPicaud}. If moreover $\Gamma$ is  cocompact, then Knieper's work \cite{MR1652924}  implies that the Rick's Bowen-Margulis measure is the unique measure of maximal entropy on the unit tangent bundle $\quotient{\Gamma}{\SX}$.

 We summarize now what is known (from the Main Theorem of \cite{LinkPicaud} and Theorem~B above) in the special case of Hadamard {\hl manifolds}: 
\begin{mainC} Let $X$ be a Hadamard manifold and $\ \Gamma<\is(X)$ a  discrete group with the fixed point of an arbitrary rank one isometry of $\XX$ in its infinite limit set.  Then  either Knieper's measure and Ricks' Bowen-Margulis measure on $\quotient{\Gamma}{\SX}$ coincide, and the geodesic flow is conservative and ergodic with respect to this measure, or
the geodesic flow is dissipative with respect to any Knieper's measure and with respect to any Ricks' Bowen-Margulis measure on $\quotient{\Gamma}{\SX}$. Moreover, in the second case it is  non-ergodic unless the considered measure is supported on a single orbit by the geodesic flow. 
\end{mainC}
Again, in the dissipative case there may be several choices for Knieper's measure and for Ricks'  Bowen-Margulis measure on $\quotient{\Gamma}{\SX}$ as both measures are constructed from a conformal density. And even if the same conformal density is used in the construction, Knieper's measure and Ricks' Bowen-Margulis measure might be different. 

Actually, in this article we will consider slightly more general classes of measures on $\quotient{\Gamma}{[\SX]}$ respectively $\quotient{\Gamma}{\SX}$: Instead of using the geodesic current associated to a conformal density for the construction we  allow for an arbitrary quasi-product geodesic current (see Section~\ref{geodcurrentmeasures} for a precise definition).

The paper is organized as follows: In Section~\ref{prelim} we  fix some notation and recall basic facts concerning Hadamard spaces; in Section~\ref{rank1prelim} the notion of rank one isometry is recalled and basic properties are listed. 

Section~\ref{rankonegroups} discusses conditions under which a subgroup $\Gamma$ of the isometry group of a proper Hadamard space $\XX$ is {\hl rank one} (that is contains a pair of independent rank one elements), and under which hypotheses the presence of a rank one geodesic of zero width in $\XX$ with extremities in the limit set of $\Gamma$ can be guaranteed. This section is of independent interest.

 In Section~\ref{dyndef} basic notions and useful facts from ergodic theory and dynamical systems are recalled, and the important notion of quasi-product geodesic current is introduced. We also recall from \cite{Ricks} Ricks' construction of a geodesic flow invariant measure associated to such a geodesic current first on the quotient $\quotient{\Gamma}{[\SX]}$ of parallel classes of parametrized geodesic lines and finally on the quotient $\quotient{\Gamma}{\SX}$ of parametrized geodesic lines. Section~\ref{propradlimset} deals with the relation between the radial limit set of the group $\Gamma$ and recurrence in $\quotient{\Gamma}{[\SX]}$ respectively $\quotient{\Gamma}{\SX}$. We deduce the crucial Theorem~\ref{zerofull}, which in particular implies that for a rank one group $\Gamma$ with the extremity of a zero width rank one geodesic in its limit set any conservative  quasi-product geodesic current $\overline\mu$ is supported on the set of end point pairs of zero width rank one geodesics. 
In Section~\ref{HopfArgument} we use the Hopf argument to show that under the presence of a {\hl zero width} rank one geodesic with extremities in the limit set  conservativity of a quasi-product geodesic current $\overline\mu$ satisfying a mild growth condition implies ergodicity of  the geodesic  flow with respect to the associated geodesic flow invarant Ricks' measure. Compared to the classical case a few technical issues need to be addressed there. 

In Section~\ref{currentsfromconfdens} we then specialize to geodesic currents coming from a conformal density. We recall a few properties of conformal densities and prove Proposition~\ref{confgivesdissipative}, which states that for convergent groups every Ricks' measure on    $\quotient{\Gamma}{[\SX]}$ is dissipative. Section~\ref{divergentmeansconservative} is devoted to the proof of Proposition~\ref{divseries}, namely that divergent groups always induce conservative Ricks' measure. The minimal requirement that $\Gamma$ contains only a rank one element of {\hl arbitrary width} makes the proof a bit more technical than it would be with the presence of a  zero width geodesic with extremities in the limit set; however,  it is needed in this form to obtain Theorem~\ref{HTSweak} which is Theorem~A above.  In the final section~\ref{conclusion} we summarize our results to deduce Theorems~A, B and C.  
Following an idea of F.~Dal'bo, M.~Peign{\'e} and J.P. Otal  (\cite{MR1776078},  \cite{MR3220550}) we also show how to construct plenty of convergent discrete rank one isometry groups of any Hadamard space admitting a rank one isometry.

\section{Preliminaries on Hadamard spaces}\label{prelim}

The purpose of this section is to introduce terminology and notation and to summarize basic results about Hadamard spaces.  
Most of the material can be found in  \cite{MR1744486} and  \cite{MR1377265} (see also \cite{MR823981}  in  the special case of Hadamard manifolds and \cite{Ricks} for more recent results). 

Let $(\XX,d)$ be a metric space. For $y\in \XX$ and $r>0$ we will denote
$B_y(r)\subseteq\XX$ the open ball of radius $r$ centered at $y\in\XX$.
A {\hd geodesic} is a map $\sigma$ from a closed interval $I\subseteq\RR$ or $I=\RR$ to $\XX$ \st $d(\sigma(t), \sigma(t'))=|t-t'|$ for all $t,t'\in I$. For more precision we use the term {\hd geodesic ray} if $I=[0,\infty)$ and  {\hd geodesic line} if $I=\RR$.

We will deal here with  {\hl Hadamard spaces} $(\XX,d)$, that is complete metric spaces in which for any two points $x,y\in\XX$ there exists a geodesic $\sigma:[0,d(x,y)]\to \XX$ with $\sigma(0)=x$ and $\sigma(d(x,y))=y$ and in which all geodesic triangles satisfy the CAT$(0)$-inequality. This implies in particular that $\XX$ is simply connected and that the geodesic joining an arbitrary pair of points in $\XX$ is unique.   Notice  however that in the non-Riemannian setting completeness of $\XX$ does not imply that every geodesic  can be extended to a geodesic line, so $\XX$ need not be geodesically complete. The geometric boundary $\rand$ of
$\XX$ is the set of equivalence classes of asymptotic geodesic
rays endowed with the cone topology (see for example Chapter~II in \cite{MR1377265}).  We remark that for all $x\in\XX$ and all $
\xi\in\rand$  there exists a unique geodesic ray $\sigma_{x,\xi}$ with origin $x=\sigma_{x,\xi}(0)$ representing $\xi$.

From here on we will require that $\XX$ is proper; in this case the geometric boundary $\rand$ is compact and the space $\XX$ is a dense and open subset of the compact space $\ganz:=\XX\cup\rand$.
Moreover, the action of the isometry group $\is(\XX)$ on $\XX$ naturally extends to an action by homeomorphisms on the geometric boundary. 

If $x, y\in \XX$, $\xi\in\rand$ and $\sigma$ is a geodesic ray in the
class of $\xi$, we set 
\begin{equation}\label{buseman}
 \bs_{\xi}(x, y)\,:= \lim_{s\to\infty}\big(
d(x,\sigma(s))-d(y,\sigma(s))\big).
\end{equation}
This number exists, is independent of the chosen ray $\sigma$, and the
function
\[ \bs_{\xi}(\cdot , y):
 \XX \to  \RR,\quad
x \mapsto \bs_{\xi}(x, y)\]
is called the {\hl Busemann function} centered at $\xi$ based at $y$ (see also Chapter~II in~\cite{MR1377265}). Obviously we have
\[ \bs_{g\cdot\xi}(g\at x,g\at y) = \bs_{\xi}(x, y)\quad\text{for all }\ x,y\in\XX\quad\text{and }\  g\in\is(\XX),\]
and the cocycle identity  
\begin{equation}\label{cocycle}
\bs_{\xi}(x, z)=\bs_{\xi}(x, y)+\bs_{\xi}(y,z)
\end{equation}
holds for all $x,y,z\in\XX$.

Since $\XX$ is non-Riemannian in general, we consider (as a substitute of the unit tangent bundle $S\XX$) the set  of parametrized geodesic lines in $\XX$ which we will denote $\SX$. We endow this set  with the metric $d_1$ given by
\begin{equation}\label{metriconSX} d_1(u,v):=\sup \{ \e^{-|t|} d\bigl(v(t), u(t)\bigr) \colon t\in\RR\}\ \mbox{ for} \ u,v\in \SX ;\end{equation}
this metric induces the compact-open topology, and every isometry of $\XX$ naturally extends to an isometry of the metric space $(\SX,d_1)$.

Moreover, there is a natural map $p:\SX\to\XX$ defined as follows: To a geodesic line  $v:\RR\to \XX$ in $\SX$ we assign its origin 
$pv:=v(0)\in\XX$.   
Notice that $p$ is proper, $1$-Lipschitz and $\is(\XX)$-equivariant; if $\XX$ is geodesically complete, then $p$ is surjective. 

For a geodesic line $v\in \SX$ we denote its extremities $v^-:=v(-\infty)\in\rand$ and  $v^+:=v(+\infty)\in\rand$ the negative and positive end point of $v$; in particular, we can define the end point map
\[ \ndpt:\SX\to \rand\times\rand,\quad v\mapsto (v^-,v^+).\]

We say that a point $\xi\in\rand$ can be joined to  $\eta\in\rand$  by a
geodesic $v\in \SX$ if   
$v^-=\xi$ and $v^+=\eta$. Obviously the  set of pairs $(\xi,\eta)\in\rand\times\rand$   \st $\xi$ and $\eta$ can be joined by a geodesic coincides with  $ \ndpt\SX $, the image of $\SX$ under the end point map $\ndpt$. It is well-known that if $\XX$ is CAT$(-1)$, then any pair of distinct boundary points $(\xi,\eta)$ belongs to $\ndpt\SX $ and  the geodesic joining $\xi$ to $\eta$ is unique up to reparametrization. In general however, the set $\ndpt\SX $ is much smaller compared to $\rand\times\rand$ minus the diagonal due to the possible existence of flat subspaces in $\XX$. 
For $(\xi,\eta)\in\ndpt\SX $ we denote by
\begin{equation}\label{joiningflat}
(\xi\eta):=p\bigl(\{ v\in \SX \colon v^-=\xi,\ v^+=\eta\}\bigr)=p\circ \ndpt^{-1}(\xi,\eta)\end{equation}
the subset of points in $\XX$ which lie on a geodesic line  joining $\xi$ to  $\eta$. It is well-known that $(\xi\eta)=(\eta\xi)\subseteq \XX$ is a closed and convex  subset of $\XX$ which is   isometric to a  product $C_{(\xi\eta)}\times\RR$, where  $C_{(\xi\eta)}=C_{(\eta\xi)}$ is again a closed and convex set. 

In order to describe the sets $(\xi\eta)$ and $C_{(\xi\eta)}$ more precisely and for later use 
 we introduce as in \cite[Definition 5.4]{Ricks} for $x\in\XX$ the so-called {\hd Hopf parametrization} map 
 \begin{equation}\label{HopfPar}
\Hopf_x: \SX\to \ndpt\SX \times \RR,\quad v\mapsto \bigl(v^-,v^+,\bs_{v^-}(x, v(0))\bigr)
\end{equation}
of $\SX$ with respect to $x$. It is immediate that for a CAT$(-1)$-space $\XX$ this map is a homeomorphism; in general  it is only continuous and surjective. 
Moreover, it depends on the point $x\in\XX$ as follows: If $y\in \XX$,
$v\in \SX$ and  $\Hopf_x(v)=(\xi,\eta,s)$, then 
\[ \Hopf_y(v)=\bigl(\xi,\eta,s+\bs_{\xi}(y,x)\bigr)\]
by the cocycle identity~(\ref{cocycle}) for the Busemann function 
(compare also \cite[Section~3]{MR1207579}). 

The Hopf parametrization map allows to define an equivalence relation $\sim$ on $\SX$ as follows: If $u,v\in \SX$, then $u\sim v$ if and only if $\Hopf_\xo(u)=\Hopf_\xo(v)$. Notice that this definition does not depend on the choice of $\xo\in\XX$ and that every point $(\xi,\eta,s)\in\ndpt\SX\times \RR$ uniquely determines an equivalence class $[v]$ with $v\in\SX$. Moreover, the closed and convex set $C_{(\xi\eta)}$ from above can be identified with the set 
\begin{equation}\label{transversal}
C_v:=p\bigl(\{u\in\SX\colon u\sim v\}\bigr)\subseteq\XX,
\end{equation}  
which we will call the  {\hl transversal}  of $v$. We remark that for all $w\in \ndpt^{-1}(\xi,\eta)$ the transversal $C_w$ is isometric to $C_v$. Moreover, if
 $\XX$ is CAT$(-1)$ then for all $v\in\SX$ the transversal $C_v$ is simply a point; in general, the transversals can be unbounded. 

As stated in \cite[Proposition 5.10]{Ricks} the $\is(\XX)$-action on $\SX$ descends to an action on $\ndpt\SX \times \RR=\Hopf_\xo(\SX)$ by homeomorphisms via 
\[ \gamma (\xi,\eta, s):=\bigl(\gamma \xi,\gamma \eta, s+\bs_{\gamma\xi}(\xo,\gamma \xo)\bigr).\]

Moreover, the action of $\is(\XX)$ is well-defined on  the set of equivalence classes  $[\SX]$ of elements in $\SX$, and the (well-defined) map
\begin{equation}\label{equivHopf} [\SX]\to \ndpt\SX\times \RR,\quad [v]\mapsto \Hopf_\xo(v)\end{equation}
is an $\is(\XX)$-equivariant homeomorphism. 
 For convenience we will frequently identify  $
\ndpt\SX\times \RR$ with  $[\SX]$. 
We also remark that the end point map $\ndpt:\SX\to \rand\times \rand$ induces a well-defined map $[\SX]\to\rand\times\rand$ which we will also denote $\ndpt$. 

As in Definition~5.4 of \cite{Ricks} we will say that a  sequence $(v_n)\subseteq\SX$ {\hd  converges weakly} to $v\in \SX$ if and only if 
  \begin{equation}\label{defweakconvergence}
  v_n^-\to v^-,\quad v_n^+\to v^+\quad\text{and }\  \bs_{v_n^-}\bigl(\xo,v_n(0)\bigr)\to \bs_{v^-}\bigl(\xo,v(0)\bigr).\end{equation}
 Obviously, weak convergence $v_n\to v$ is equivalent to the convergence $[v_n]\to [v]$ in $[\SX]$, and  $v_n\to v$ in $\SX$ always implies $[v_n]\to [v]$ in $[\SX]$.

 The topological space $\SX$ can be endowed with the {\hd geodesic flow} $(g^t)_{t\in\RR}$  which  is naturally defined by reparametrization of $v\in \SX$. In particular we have
\[ (g^t v)(0)=v(t) \quad\text{for all } \ v\in \SX\quad\text{and all }\ t\in\RR.\]
The geodesic flow induces a flow on the set of equivalence classes $[\SX]$ which we will also denote $(g^t)_{t\in\RR}$; via the $\is(\XX)$-equivariant homeomorphism $[\SX]\to\ndpt\SX\times \RR$ the action of the geodesic flow $(g^t)_{t\in\RR}$ on $[\SX]$ is equivalent to the translation action on the last factor of $\ndpt\SX \times \RR$ given by
\[ g^t (\xi,\eta,s):=(\xi,\eta, s+t).\]

\section{Facts about rank one isometries}\label{rank1prelim}

The purpose of this section is to introduce the notion of  rank one geodesic and rank one  isometry. Many useful well-known facts about Hadamard spaces with a rank one isometry are recalled. Most of the material can be found in 
\cite{MR1377265} and \cite{MR1383216} (see also \cite{MR656659} for the special case of Hadamard manifolds  and \cite{Ricks} for more recent results). 

As in the previous section we assume that $(X,d)$ is a proper Hadamard space.
A geodesic line $v\in \SX$  
 is called  
{\hl rank one} if its transversal $C_v$ is bounded. In this case the number
\[ \width(v):= \sup\{ d(x,y)\colon x,y\in C_v\}\]
is called the {\hl width} of $v$; if   $C_v$ reduces to a point, then $v$ is said to have zero width. 
In the sequel we will use as in \cite{Ricks} the notation 
\begin{align*}
\mathcal{R}&:=\{v\in \SX\colon v\ \text{is rank one}\}\quad\text{respectively}\\
\mathcal{Z}&:=\{v\in \SX\colon v\ \text{is rank one of zero width}\}.
\end{align*}

We remark that the existence of a rank one geodesic  imposes severe restrictions on the Hadamard space $\XX$. For example,  $\XX$ can neither be a symmetric space or  Euclidean building of higher rank  nor a product of Hadamard spaces.

Notice that if $\XX$ is a Hadamard {\hl manifold}, then there is a more restrictive notion of rank one: If $v\in\SX$ the number $  J$-rank$(v)$ is defined as the dimension of the vector space  of parallel Jacobi fields along $v$ (compare Section~IV.4 in \cite{MR1377265}); clearly, for all $w$ in a sufficiently small neighborhood of $v$ we have    $J$-rank$(w)\le J$-rank$(v)$.  As in \cite{LinkPicaud} we will call $v\in\SX$  {\hl strong rank one}  if $  J$-rank$(v)=1$, that is if $v$ does not admit a parallel perpendicular Jacobi field; we further define
\[\jac:=\{v\in\SX\colon v \ \text{is strong rank one}\}\]
which is obviously a subset of $\zero$. Notice that in general $\jac\ne \zero$: Take for example a surface with negative Gau\ss ian curvature except along a simple closed geodesic where the curvature vanishes; then the lift of the closed geodesic has zero width, but possesses a parallel perpendicular Jacobi field.

The following important lemma states that even though we cannot join any two distinct points in the geometric boundary $\rand$ of the Hadamard space $\XX$, given a rank one geodesic we can at least join all points in a neighborhood of its end points. More precisely,  we have the following result which is a reformulation of  Lemma~III.3.1 in \cite{MR1377265}:
\begin{lemma}[Ballmann]\label{joinrankone} \
Let $v\in\mathcal{R}$ be a rank one geodesic and $c>\width(v)$.
Then there exist open disjoint neighborhoods $U^-$ of $\,v^-$ and $U^+$ of $\,v^+$ in $\ganz$ with the following properties:
If $\xi\in U^-$ and $\eta \in U^+$ then there exists a rank one geodesic joining $\xi$ and $\eta$. For any such geodesic $w\in\mathcal{R}$ we have   $d(w(t), v(0))< c$ for some $t\in\RR$ and $\width(w)\le 2c$.
\end{lemma}
This lemma implies that the set $\reg$ is open in $\SX$; we emphasize that $\zero$ in general need not be an open subset of $\SX$: In every open neighborhood of a {\hl zero width} rank one geodesic there may exist a rank one geodesic of arbitrarily small but strictly positive width. However, if $\XX$ is a Hadamard {\hl manifold}, then $\jac\subseteq\zero$  
is open in $\SX$ (as the $J$-rank cannot be bigger in a suffiently small open neighborhood). 
So Lemma~\ref{joinrankone} has the following
\begin{corollary}\label{manifoldJopen}
Let $v\in\jac$. Then there exist disjoint neighborhoods $U^-$ of $\,v^-$ and $U^+$ of $\,v^+$ in $\ganz$ such that any pair of points $(\xi,\eta)\in U^-\times U^+$ can be joined by a geodesic $u\in\jac$.  \end{corollary}

We will also need the following 
result due to R.~Ricks;  
recall that $(v_n)\to v$ weakly as defined in (\ref{defweakconvergence}) means that $[v_n]\to [v]$ in $[\SX]$.
\begin{lemma}[\cite{Ricks}, Lemma~5.9]\label{weakimpliesstrong}
If a sequence $(v_n)\subseteq\SX$ converges weakly to \underline{$v\in\reg$}, then some subsequence of $(v_n)$ converges to some $u\sim v$. 
\end{lemma}
Notice that this lemma implies that the  restriction of the Hopf parametrization map~(\ref{HopfPar})   to the subset $\mathcal{R}$ is closed, hence a topological quotient map.

In combination with Lemma~8.4 in \cite{Ricks} we get
the following statement concerning transversals of a weakly convergent sequence in $\SX$:
\begin{lemma}\label{Hausdorffconv}
If a sequence $(v_n)\subseteq\SX$ converges weakly to $v\in\reg$, then some subsequence of $(C_{v_n})$ converges,  in the Hausdorff metric, to a closed subset $A\subseteq C_v$.
\end{lemma}
From this we immediately get the following complement to Lemma~\ref{joinrankone}:
\begin{lemma}\label{Hausdorffonboundary}
Let $v\in\zero$ and $\bigl((\xi_n,\eta_n)\bigr)\subseteq\rand\times\rand$ be a sequence converging to $(v^-,v^+)$. Then for $n$ sufficiently large $(\xi_n,\eta_n)\in\ndpt\reg$ and some subsequence of $\bigl(C_{(\xi_n\eta_n)}\bigr)$ converges,  in the Hausdorff metric, to a point. 
\end{lemma}
\begin{definition}\label{hypaxiso}\
An isometry $\gamma\neq\id$ of $\XX$ is called {\hd axial}  if there exists a constant
$\ell=\ell(\gamma)>0$ and a geodesic $v\in \SX$ \st $\gamma v=g^{\ell} v$. 
We call
$\ell(\gamma)$ the {\hd translation length} of $\gamma$, and $v$
an {\hd invariant geodesic} of $\gamma$. The boundary point
$\gamma^+:=v^+$ (which is independent of the chosen invariant geodesic $v$) is called the {\hd attractive fixed
point}, and $\gamma^-:=v^-$ the {\hd repulsive fixed
point} of $\gamma$. 

An axial isometry $h$ is called {\hd rank one} if one (and hence any) invariant geodesic of $h$ belongs to  $\reg$; the  {\hd width} of $h$ is then defined as the width of an arbitrary invariant geodesic of $h$. $h$ is said to have {\hd zero width}  if up to reparametrization $h$ has only one invariant geodesic.
\end{definition}
Notice that if $\gamma\in\is(\XX)$ is axial, then $\partial^{-1}(\gamma^-,\gamma^+)\subseteq\SX$ is the set of parametrized invariant geodesics of $\gamma$, and every axial isometry $\widetilde\gamma$ commuting with $\gamma$ satisfies $p \partial^{-1}(\widetilde\gamma^-,\widetilde\gamma^+)=p \partial^{-1}(\gamma^-,\gamma^+)$.  If $h$ is rank one, then  the fixed point set of $h$ equals $\{h^-, h^+\}$, and every axial isometry  commuting with $h$ belongs to the subgroup $\langle h\rangle<\is(\XX)$ generated by $h$. 

The following important lemma describes the north-south dynamics of rank one isometries:
\begin{lemma}\label{dynrankone}(\cite{MR1377265}, Lemma III.3.3)\
\ Let $h$ be a rank one isometry. Then
\begin{enumerate}
\item[(a)]  every point $\xi\in\rand\setminus\{h^+\}$ can be joined
to $h^+$ by a geodesic, and all these geodesics are  rank one,
\item[(b)] given neighborhoods $U^-$ of $h^-$ and $U^+$ of $h^+$ in $\ganz$
there exists $N\in\NN$ \st
 $\  h^{-n}(\ganz\setminus U^+)\subseteq U^-$ and
$h^{n}(\ganz\setminus U^-)\subseteq U^+$ for all $n\ge N$.
\end{enumerate}
\end{lemma}

The following lemma shows that under the presence of a rank one geodesic in $\XX$ with {\hl $\is(\XX)$-dual} end points (the interested reader is referred to Section~III.1 in \cite{MR1377265} for a definition) the rank one isometries are numerous:  
\begin{lemma}\label{elementsarerankone} (\cite{MR1377265}, Lemma III.3.2)\
Let $v\in \reg$ be a rank one geodesic, and $(g_n)\subseteq\is(\XX)$ a sequence of isometries \st $g_n x\to v^+$ and $g_n^{-1}x\to v^-$ for one (and hence any) $x\in\XX$. Then, for $n$ sufficiently large, $g_n$ is rank one with an invariant geodesic  $v_n$ \st $v_n^+\to v^+$ and $v_n^-\to v^-$.
\end{lemma}
We next prepare for an extension of Lemma~\ref{dynrankone} (a) which replaces the fixed point $h^+$ of the rank one isometry $h$ by the end point of a certain geodesic:
\begin{definition}[compare Section~5 in \cite{Ricks}]\label{weakstrongrecurrencedef} Let $G<\is(\XX)$ be any subgroup. An element   $v\in\SX$ is said to {\hd (weakly) $G$-accumulate} on $u\in\SX$ if  
there exist sequences $(g_n)\subseteq G$ and $(t_n)\nearrow \infty$  \st $g_n g^{t_n} v$ converges (weakly) to $u$ as $n\to\infty$; $v$ is said to be {\hd (weakly) $G$-recurrent} if $v$ (weakly) $G$-accumulates on $v$. 
 \end{definition} 
Notice that if $v$ is an invariant geodesic of an axial isometry $\gamma\in\is(\XX)$, then $v$ is $\langle \gamma\rangle$-recurrent and hence in particular $\is(\XX)$-recurrent. Moreover, if $v\in\SX$ weakly $G$-accumulates on  \underline{$u\in\reg$}, then by Lemma~\ref{weakimpliesstrong} $v$ $G$-accumulates on some element $w\sim u$. However, in general $v\in\SX$ weakly $G$-recurrent  does \underline{not} imply that some representative of the equivalence class $[v]$ is $G$-recurrent. Even in the case $v\in\reg$ it is possible that every representative $u$ of the class $[v]$  $G$-accumulates on $w\sim u$ with $w\ne u$. 

The following statements show the relevance of the previous notions. 
\begin{lemma}[\cite{Ricks}, Lemma~6.10]\label{Gammaconv}
If $w\in\SX$ $\is(\XX)$-accumulates on $v\in\SX$, then  there exists an isometric embedding $C_{w}\hookrightarrow C_v$ which maps $w(0)$ to $v(0)$. 
\end{lemma}
Notice that if \underline{$v\in\reg$} is weakly $G$-recurrent for some subgroup $G<\is(\XX)$, then every $w\in\SX$ with $w^+=v^+$ $G$-accumulates on an element $u\sim v$  according to Lemma~6.9 in \cite{Ricks}.  Hence we have
\begin{lemma}[Corollary~6.11 in \cite{Ricks}]\label{weakgivesisometricembeddings}
If $v\in\reg$ is weakly $\is(\XX)$-recurrent, then for every $w\in\SX$ with $w^+=v^+$ there exists an isometric embedding $C_w\hookrightarrow C_v$.
\end{lemma}
Moreover, the 
proof of Lemma~6.12 in \cite{Ricks} shows that every point $\xi\in\rand\setminus\{v^+\}$ can be joined to $v^+$ by a geodesic $w\in\SX$. So we finally get
\begin{lemma}\label{jointoweakrecurrent}
If $v\in\reg$ is weakly $\is(\XX)$-recurrent then for every $\xi\in\rand\setminus\{v^+\}$ there exists $w\in\reg$ with $\width(w)\le \width(v)$ such that $w^-=\xi$ and $w^+=v^+$. \end{lemma}

\section{Rank one groups}\label{rankonegroups}

Let $\XX$ be a proper
 Hadamard space and $\Gamma<\is(\XX)$ an arbitrary subgroup. 
 The {\hl geometric limit set} $\Lim$ of $\Gamma$  is defined by
$\Lim:=\overline{\Gamma\cdot x}\cap\rand,$
where $x\in\XX$ is an arbitrary point.  

If $\XX$ is a CAT$(-1)$-space, then a  group $\Gamma<\is(\XX)$ is called {\hl non-elementary} if its limit set is infinite and if $\Gamma$ does not globally fix a point in $\Lim$. It is well-known that this implies that $\Gamma$ contains two axial isometries  with disjoint fixed point sets (which are actually rank one of zero width as $\SX=\zero$ for any CAT$(-1)$-space). In the general setting this motivates the following 
\begin{definition}
We say that  two rank one isometries $g,h\in\is(\XX)$ are {\hd independent} if and only if $\{g^+,g^-\}\cap \{h^+,h^-\}\ne\emptyset$ (see for example Section~2 of \cite{MR2629900}).  

Moreover, 
 a group $\Gamma< \is(\XX)$ is called  {\hl rank one} if $\Gamma$ contains a pair of independent rank one elements.
\end{definition}
Obviously, if $\XX$ is CAT$(-1)$ then every non-elementary isometry group is rank one. In general however, the notion of rank one group seems very restrictive at first sight.  The goal of this section -- which may be of independent interest --  is to  discuss  conditions which ensure that $\Gamma$ is a rank one group.

\begin{lemma} Let $\Gamma < \is(\XX)$ be an arbitrary subgroup. If 
$\Lim$ contains the positive end point $v^+$  of a weakly $\is(\XX)$-recurrent element $v\in\reg$, and if $v^+$ 
is not globally fixed by $\Gamma$, then $\Gamma$ contains a  rank one isometry. 
\end{lemma}
\begin{proof}
Let $v\in\reg$ be weakly $\is(\XX)$-recurrent and $x\in\XX$. As  $v^+\in \Lim$ there exists a sequence $(\gamma_n)\subseteq\Gamma$ such that $\gamma_n x\to v^+$ as $n\to\infty$. Passing to a subsequence if necessary we may assume that $\gamma_n^{-1}x$ converges, say to a point $\xi\in\ganz$ which obviously belongs to $\Lim\subseteq\rand$. If $\xi=v^+$, there exists $\gamma\in\Gamma$ \st $\gamma\xi\ne v^+$ since $\Gamma$ does not globally fix $v^+$. Replacing the sequence $(\gamma_n)$ by $(\gamma_n\gamma^{-1})$ in this case we may assume that $\xi\ne v^+$. According to Lemma~\ref{jointoweakrecurrent} there exists $w\in\reg$ \st $w^-=\xi$ and $w^+=v^+$. 
Lemma~\ref{elementsarerankone} then states that for $n$ sufficiently large $\gamma_n$ is rank one with an invariant geodesic  $v_n$ \st  $v_n^+\to w^+=v^+$ and $v_n^-\to w^-=\xi$ as $n\to \infty$. Since the geodesic $w$ is rank one, the geodesics $v_n$ are rank one  for $n$ sufficiently large by Lemma~\ref{joinrankone}. This implies that for some fixed $n$  large enough the element $\gamma_n\in \Gamma$ is rank one.
\end{proof}
Notice that the conclusion is obviously true when $v^+$ is a fixed point of a rank one isometry of $\XX$. The following statements show that a group is rank one under very weak conditions.
\begin{lemma}\label{get2independent} If $\Gamma< \is(\XX)$ neither globally fixes a point in $\rand$  nor stabilizes a geodesic line in $\XX$, and if 
$\Lim$ contains the positive end point $v^+$  of a weakly $\is(\XX)$-recurrent element $v\in\reg$, then $\Gamma$  contains a pair of independent rank one elements. 
\end{lemma}
\begin{proof}
Since $\XX$ is proper and $\Gamma <\is(\XX)$ contains a  rank one element by the previous lemma, Proposition~3.4 of \cite{MR2585575} applies: Its first possibility is excluded by the assumption that $\Gamma$ neither globally fixes a point in $\rand$  nor stabilizes a geodesic line in $\XX$, hence $\Gamma$ contains a pair of independent rank one elements. 
\end{proof}
\begin{lemma}\label{inflimset}
A {\em discrete}  
subgroup $\Gamma<\is(\XX)$ is rank one  if and only if 
its limit set $\Lim$ is infinite and contains the positive end point $v^+$  of a weakly $\is(\XX)$-recurrent element $v\in\reg$. 
\end{lemma}
\begin{proof} We first assume that  $\Lim$ is infinite and contains the positive end point $v^+$  of a weakly $\is(\XX)$-recurrent element $v\in\reg$. As $\Gamma$ is discrete and $\Lim$ is infinite, $\Gamma$ cannot globally fix a point in $\rand$ nor stabilize a geodesic line in $\XX$, so  Lemma~\ref{get2independent} above implies that $\Gamma$ is rank one. 
The other direction is obvious.
\end{proof}

The proof of the following criterion relies heavily on the work of  
R.~Ricks:
 \begin{proposition}\label{largewidthgiveszerowidth}
 If $\XX$ is geodesically complete and $\Gamma<\is(\XX)$ is a discrete rank one group,  
 then 
\[ {\mathcal Z}_\Gamma:=\{v\in\zero\colon v^-,v^+\in\Lim\}\ne \emptyset.\] \end{proposition}
 \begin{proof} 
We first notice that the proof of Theorem~III.2.3 in \cite{MR1377265} shows that the geodesic flow restricted to 
\[ \SX_\Gamma:=\{v\in \SX\colon v^-,v^+\in\Lim\}\]
 is topologically transitive mod $\Gamma$; this means that there exists $v\in \SX_\Gamma$ \st for any $u\in \SX_\Gamma$ $v$ $\Gamma$-accumulates on $u$. 
  
We first claim that the element $v\in \SX_\Gamma$ as above belongs to $\reg$:
We choose a rank one element $h\in\Gamma$ and an invariant geodesic $u\in \SX_\Gamma$  of $h$  and neighborhoods $U^-, U^+\subseteq \ganz$  of $h^-,h^+$ as in Lemma~\ref{joinrankone}. In particular, every $w\in\SX$ with $(w^-,w^+)\in U^-\times U^+$ satisfies $w\in\reg$.  As $v$ $\Gamma$-accumulates on $u$ there exist sequences  $(\gamma_n)\subseteq \Gamma$, $(t_n)\nearrow\infty$ \st $\gamma_n g^{t_n} v\to u$ and hence in particular 
$\gamma_n ( v^-,v^+)\to (u^-,u^+)=(h^-,h^+)$ as $n\to \infty$. This implies $\gamma_n (v^-,v^+)\in U^-\times U^+\subseteq\ndpt\reg$ for some $n$ sufficiently large and therefore $v\in \reg$.

Assume for a contradiction that $v\notin\zero$; then there exists $\overline{v}\sim v$ with $\overline{v}\neq v$. We will further denote $v_C\in p^{-1} C_v=\{w\in\reg\colon w\sim v\}$  the {\hl central geodesic}  defined by the condition that its origin $v_C(0)$ is the unique circumcenter of the bounded closed and convex set  $C_v\subseteq \XX$ (compare also Section~5 in \cite{Ricks}). 
As $v_C$, $\overline{v}\in\SX_\Gamma$, $v$\break $\Gamma$-accumulates both on $v_C$ and on $\overline{v}$; 
so  according to Lemma~\ref{Gammaconv} there exist isometric embeddings
  \[  \iota: C_v \hookrightarrow C_{v_C},\qquad \overline{\iota}: C_v \hookrightarrow C_{\overline{v}} \]
  with $\iota\bigl(v(0)\bigr)= v_C(0)$ and $\overline{\iota}\bigl(v(0)\bigr)=\overline{v}(0)$. Since $C_{v_C} =C_{\overline{v}}=C_v$, the maps $\iota$ and $\overline{\iota}$ 
  are surjective by Theorem~1.6.15 in \cite{MR1835418} and hence isometries. As the circumcenter of $C_v$ is invariant by isometries of $C_{v}$ we first get 
  \[ v(0)=\iota^{-1}\bigl(v_C(0)\bigr)=v_C(0),\] 
which implies
\[   \overline{v}(0)= \overline{\iota}\bigl(v(0)\bigr)= \overline{\iota}\bigl(v_C(0)\bigr)=v_C(0)=v(0).\]
This is a contradiction to the choice of $\overline{v}\ne v$, so we conclude that  $v\in\zero$. 
\end{proof}
Notice that a discrete rank one group $\Gamma$ with $\zero_\Gamma\ne\emptyset$ need not possess a {\hl zero width} rank one {\hl isometry} since $\zero$ is not open in $\SX$. However, as for a Hadamard {\hl manifold} the set $\jac$ of vectors not admitting a parallel perpendicular Jacobi field is open in $\SX$, we have the following
\begin{lemma}
If $\XX$ is a {\hl manifold} and $\Gamma<\is(\XX)$ a discrete rank one group \st 
\[ \jac_\Gamma:=\{v\in\jac\colon v^-,v^+\in\Lim\}\ne \emptyset,\]
then $\Gamma$ contains a pair of independent rank one elements with {\hl strong} rank one invariant geodesics (which necessarily have zero width).
\end{lemma} 
\begin{proof} Since $\XX$ is geodesically complete, the geodesic flow restricted to 
\[ \SX_\Gamma:=\{v\in \SX\colon v^-,v^+\in\Lim\}\]
 is topologically transitive mod $\Gamma$; this means that there exists $v\in \SX_\Gamma$ \st for any $u\in \SX_\Gamma$ $v$ $\Gamma$-accumulates on $u$. 
 Assume for a contradiction that $v\notin\jac$; then $\gamma g^t v\notin \jac$ for all $\gamma\in\Gamma$ and for all $t\in\RR$. But since $v$ $\Gamma$-accumulates on $u\in \jac_\Gamma$ 
 this implies $J$-rank$(u)\ge 2$ which is a contradiction. So we conclude that $v\in\jac_\Gamma$.
 
Since  $v^-,v^+\in\Lim$, 
 there exists a sequence $(\gamma_n)\subseteq\Gamma$ \st $\gamma_n x\to v^+$ and $\gamma_n^{-1} x\to v^-$ for some $x\in\XX$ (see for example the proof of Proposition~3.5 in \cite{MR2585575}). By Lemma~\ref{elementsarerankone}, for $n$ sufficiently large  $\gamma_n$ is rank one with invariant geodesic $v_n$ \st $(v_n^-, v_n^+)\to (v^-,v^+)$ as $n\to \infty$. So according to Corollary~\ref{manifoldJopen} we have  $v_n\in \jac$ for $n$ sufficiently large, hence there exists a rank one element $\gamma_n$ with a strong rank one invariant geodesic. As $\Gamma$ is rank one  there exists one element (actually an infinite number) in $\Gamma$ not commuting with $\gamma_n$, and  conjugating $\gamma_n$ by such an element provides another rank one isometry  in $\Gamma$ independent from $\gamma_n$ which also has a strong rank one invariant geodesic. \end{proof}
 
 This implies that the hypothesis of the Main Theorem in \cite{LinkPicaud} is satisfied for Hadamard manifolds $\XX$ with a rank one group $\Gamma<\is(\XX)$ such that $\jac_\Gamma\ne\emptyset$; we will see later that the conclusion of  the Main Theorem in \cite{LinkPicaud} remains true under the weaker condition that    $\Gamma<\is(\XX)$ is an arbitrary rank one group. 

\section{Basic notions in ergodic theory and geodesic currents }\label{dyndef}\label{geodcurrentmeasures}

In this section we want to  recall a few general notions from topological dynamics and ergodic theory which will be needed later; our main references here are \cite{Hopf} and \cite{MR1293874}.  

Let $\Omega$ be a locally compact and $\sigma$-compact  
Hausdorff topological space and $\varphi$ a {\hd flow} on $\Omega$, that is a continuous map $\varphi :\RR\times \Omega\to \Omega$  \st $\varphi(0,\omega)=\omega$ and $\varphi\bigl(s,\varphi(t,\omega)\bigr)=\varphi(s+t,\omega)$ for all $s,t\in\RR$ and all $\omega\in \Omega$.

A point $\omega\in\Omega$ is said to be {\hd positively recurrent} respectively {\hd negatively recurrent} if there exists a sequence $(t_n)\nearrow\infty$ of real numbers \st 
\[ \varphi^{t_n}\omega=\varphi(t_n,\omega)\to \omega\quad\text{respectively }\  \varphi^{-t_n}\omega=\varphi(-t_n,\omega)\to \omega ;\] 
$\omega\in\Omega$ is said to be {\hd positively divergent} respectively {\hd negatively divergent} if 
for every compact set $K\subseteq \Omega$ there exists a constant $T>0$ \st for all $  t\ge T$  
\[ \varphi^t\omega=\varphi(t,\omega)\notin K\quad\text{respectively }\  \varphi^{-t}\omega=\varphi(-t,\omega)\notin K. \] 

Assume that $M$ is a Borel measure on $\Omega$ invariant by the flow $\varphi$. Then the Hopf decomposition theorem (see for instance  \cite[Theorem~3.2]{MR797411},\cite[Satz~13.1]{Hopf} ) asserts that the space $\Omega$  decomposes  into a disjoint union of $\varphi$-invariant Borel sets $\Omega_C$ and $\Omega_D$ which satisfy the following properties:
\begin{itemize}
\item[(C)] There does not exist a Borel subset $E\subseteq \Omega_C$ with $M(E)>0$ and such that the sets $\bigl( \varphi^k(E)\bigr)_{k\in\ZZ}$ are pairwise disjoint.
\item[(D)] There exists a Borel set $W\subseteq \Omega_D$ \st $\Omega_D$ is the disjoint union of sets $(W_k)_{k\in\ZZ}$, where each $W_k$ is a translate of $W$ under the flow $\varphi$.
\end{itemize}
According to Poincar{\'e}'s recurrence theorem (see for example \cite[Satz~13.2]{Hopf})\break 
\mbox{$M$-almost} every point of $\Omega_C$ is positively  recurrent. On the other hand, by Hopf's divergence theorem (see again \cite[Satz~13.2]{Hopf}), $M$-almost every point of $\Omega_D$ is positively  divergent.  This implies in particular that the sets $\Omega_C$ and $\Omega_D$ are unique up to sets of measure zero.  

The dynamical system $(\Omega,\varphi,M)$ is said to be {\hd  conservative} if $M(\Omega_D)=0$, and {\hd dissipative} if $M(\Omega_C)=0$. Notice that if the measure $M$ is finite, then due to (D) above $(\Omega,\varphi,M)$ is   conservative. Moreover, since the decomposition is the same for $\varphi^1$ and for $\varphi^{-1}$, Poincar{\'e}'s recurrence theorem and Hopf's divergence theorem imply that
$M$-almost every point of $\Omega_C$ is positively and negatively recurrent, and $M$-almost every point of $\Omega_D$ is positively and negatively divergent. Moreover, if $\rho\in\LL^1(M)$ is $M$-almost everywhere strictly positive, then -- up to a set of measure zero -- the conservative part $\Omega_C$ can be written as
\[\Omega_C=\{ \omega\in\Omega\colon \int_{0}^\infty \rho(\varphi^t \omega)\d t=\infty\}.\]

Finally, the dynamical system $(\Omega,\varphi,M)$ is called {\hd  ergodic} if every $\varphi$-invariant Borel set $E\subseteq\Omega$ either satisfies $M(E)=0$ or $M(\Omega\setminus E)=0$. Hence if a dynamical system $(\Omega,\varphi,M)$ is ergodic, then it is either conservative or dissipative; the second possibility can only occur for an infinite 
measure $M$ which is supported on a single orbit
\[ \{ \varphi^t \omega \colon t\in\RR\}\quad\text{with }\ \omega\in \Omega.\]

In Section~\ref{HopfArgument} we will need the following generalization of the Birkhoff ergodic theorem which is stated and proved on p.~53 in \cite{Hopf}:
\begin{theorem}[Hopf's individual ergodic theorem]  \label{Hopfindividual}
Assume that $(\Omega,\varphi,M)$ is conservative, and let $\rho\in \LL^1(M)$ be a function which is strictly positive   $M$-almost everywhere. 

Then for any function $f\in \LL^1(M)$  the limits
$$  f^\pm(\omega)=\lim_{T\to +\infty} \frac{\int_0^T  f(\varphi^{\pm t}(\omega))\d t}{\int_0^T  \rho(\varphi^{\pm t}(\omega))\d t}$$
exist and are equal for $M$-almost every $\omega\in\Omega$. Moreover, the functions $f^+, f^-$ are measurable and flow invariant, $\rho\cdot f^+, \rho\cdot  f^-\in\LL^1(M)$, and for every bounded measurable flow-invariant function $h$ we have \[ \int_\Omega \rho(\omega) f^\pm (\omega) h(\omega)\d M(\omega)=  \int_\Omega  f (\omega) h(\omega)\d M(\omega).\]
Finally, $(\Omega,\varphi,M)$ is ergodic if and only if for 
every function $f\in \LL^1(M)$ the associated limit function $f^+$ is constant $M$-almost everywhere.
\end{theorem}

We now want to recall the concept of geodesic current introduced for example in   \cite{MR1293874}.  From here on we let  $\XX$ be a proper Hadamard space and $\Gamma<\is(\XX)$ a discrete group. We will also use the notation introduced in Section~\ref{prelim} and Section~\ref{rank1prelim}. The geodesic flow on the quotient $\quotient{\Gamma}{\SX}$ will be denoted $g_\Gamma=(g_\Gamma^t)_{t\in\RR}$. 

Recall that a Borel measure 
 on a locally compact Hausdorff space is called {\hl Radon} if it is finite for all compact subsets.
 \begin{definition}[compare Definitions~2.3 and 2.5 in \cite{MR1293874}]\label{geodcurrent}\ \\
 A {\hd geodesic current} on $\quotient{\Gamma}{\XX}$ is a $\Gamma$-invariant Radon 
 measure on  
 $\ndpt \SX\subseteq \rand\times \rand $.
 A geodesic current $\overline\mu\,$ is said to be a {\hd quasi-product geodesic current}, if there exist probability measures $\mu_-$, $\mu_+$ on $\rand$ \st  $\overline\mu\,$ is absolutely continuous with respect to the product measure $\mu_-\otimes \mu_+$.  \end{definition} 
 A geodesic current $\overline \mu\,$ hence yields a dynamical system $(\ndpt\SX, \Gamma, \overline\mu)$ which is closely related to the dynamical system $(\rand\times\rand, \Gamma, \mu_-\otimes\mu_+)$ with the diagonal action of $\Gamma$ on $\rand\times\rand$. 
 As in \cite[p.17]{MR2057305} a Borel set $W\subseteq\ndpt\SX$ is called {\hd wandering} if for $\overline\mu$-almost every $(\xi,\eta)\in W$ the number 
 \[ \#\{\gamma\in\Gamma\colon \gamma (\xi,\eta)\in W\}\quad\text{is finite}.\]
  The $\Gamma$-action on $\ndpt\SX$ is called {\hd dissipative} if up to sets of measure zero the set $\ndpt\SX$ is a countable union of wandering sets; it is called {\hd conservative} if every wandering subset $W\subseteq\ndpt\SX$ satisfies $\overline\mu(W)=0$.
  
Let $\overline \mu\,$ be a geodesic current 
such that for $\overline\mu$-almost every $(\xi,\eta)\in\ndpt\SX$ a geodesic flow invariant Radon measure $\lambda_{(\xi\eta)}$ on the closed and convex subset $(\xi\eta)\subseteq\XX$ exists. Then we get a $\Gamma$-invariant and geodesic flow invariant Borel measure $m$ on $\SX$ by integrating $\overline\mu\,$ with respect to the measure $\lambda_{(\xi\eta)}$ along the sets $(\xi\eta)\subseteq\XX$, that is via the assignment
\[ m(E):= \int_{\ndpt\SX} \lambda_{(\xi\eta)}\bigl(p(E)\cap(\xi\eta)\bigr)\mathrm{d}\overline\mu(\xi,\eta)\quad\text{for any Borel set }\  E\subseteq \SX.\]
Notice that 
by continuity of the maps $p:\SX\to\XX$ and  $\ndpt:\SX\to\ndpt\SX$ the Borel measure $m$ is Radon as well. If $(\xi,\eta)\in\ndpt\zero$, then we use the convention that the  Radon measure $\lambda_{(\xi\eta)}$ on $(\xi\eta)\cong\RR$ 
is Lebesgue measure on $\RR$ (which in addition is inner and outer regular). 

The Radon measure $m$ then induces a geodesic flow invariant measure $m_\Gamma$ on the quotient $\quotient{\Gamma}{\SX}$ which we will call a {\hd Knieper's measure} on $\quotient{\Gamma}{\SX}$ for the following reason: In \cite{MR1652924}, G.~Knieper constructed for a Hadamard {\hl manifold} $\XX$ a measure on $\quotient{\Gamma}{\SX}$ precisely in this way with $\lambda_{(\xi\eta)}$ the induced Riemannian volume element on the submanifolds $(\xi\eta)\subseteq\XX$ and $\overline\mu\,$ the quasi-product geodesic current induced by a conformal density for $\Gamma$ (see Section~\ref{currentsfromconfdens} for the precise definition).

Unfortunately, if $\XX$ is not a manifold then in general there is no natural geodesic flow invariant measure on the closed and convex subsets $(\xi\eta)$ for $(\xi,\eta)\in\ndpt(\SX\setminus\zero)$.
Hence we will follow Ricks' approach to obtain from a geodesic current a geodesic flow and $\Gamma$-invariant measure on the set of parallel classes of parametrized geodesic lines $[\SX]$: Given a geodesic current $\overline \mu\,$  on $\ndpt\SX=\ndpt[\SX]$ we want to define a Radon 
measure $\overline m$ on $[\SX]\cong \ndpt\SX\times\RR$ by $\overline\mu\otimes\lambda$, where $\lambda$ denotes Lebesgue measure on $\RR$. 

However, the $\Gamma$-action on $[\SX]$ need not be proper: If $\Gamma$ contains  an axial isometry $\gamma$ with invariant geodesic $w\in \SX\setminus\reg$ whose image $w(\RR)$ belongs to an isometric copy $E\subset(\gamma^-\gamma^+)$ of a Euclidean plane, then for any geodesic $u\in\SX$ orthogonal to $w$ and with image $u(\RR)\subseteq E$  
we have $\gamma^k u\sim u$  and hence $\gamma^k [u]=[u]$ for all $k\in\ZZ$.  So in particular we do not necessarily obtain from $\overline m$ a geodesic flow invariant measure on the quotient $\quotient{\Gamma}{[\SX]}$. For that reason we will consider only geodesic currents $\overline\mu\,$ which are defined on $\ndpt\reg$ instead of $\ndpt\SX$.  

According to Lemma~\ref{joinrankone}, $\Gamma$ acts  properly on $[\reg]\cong\ndpt\reg\times \RR$ which  admits a proper metric. Since the action is by homeomorphisms and  preserves the Borel  measure $\overline m=\overline\mu\otimes\lambda$, there is (see for instance, \cite[Appendix A]{RicksThesis}) a unique Borel quotient measure $\overline m_\Gamma$ on $\quotient{\Gamma}{[\reg]}$ 􏱂 satisfying the characterizing property
\[ \int_{\bar A} \widetilde h\d \overline m=\int_{\quotient{\Gamma}{[\reg]}} \bigl( h\cdot f_{\bar  A}\bigr) \d \overline m_\Gamma\]
for all Borel sets  $ \bar A\subseteq [\reg]$ and $\Gamma$-invariant Borel maps 
$ \widetilde h:[\reg]\to [0,\infty]$ and\break $\widetilde f_{\bar A}:[\reg]\to [0,\infty]$ defined by $\widetilde f_{\bar A}([v]):= \#\{\gamma\in\Gamma\colon \gamma [v]\in \bar A\}$ for $[v]\in\reg$,  and with   $h$ and $f_{\bar A}$ the maps on $\quotient{\Gamma}{[\reg]}$ induced from $\widetilde h$ and $\widetilde f_{\bar A}$.   

According to the characterizing property above, a Borel set $\bar A\subset[\SX]$ satisfies $\overline m(\bar A)=0$ if and only if its projection $\bar A_\Gamma$ to $\quotient{\Gamma}{[\SX]}$ satisfies $m_\Gamma(\bar A_\Gamma)=0$. So in fact we can consider $\overline m_\Gamma$ as a Borel measure on $\quotient{\Gamma}{[\SX]}$; we will call  $\overline m_\Gamma$ the {\hd weak Ricks' measure} associated to the geodesic current $\overline\mu\,$  on $\ndpt\reg$.

Our final goal is to construct from a weak Ricks' measure $\overline m_\Gamma$ a geodesic flow invariant measure on $\quotient{\Gamma}{\SX}$. 
So let us first remark that $\zero\subseteq\reg$ is a Borel subset  by  semicontinuity (see Lemma~\ref{Hausdorffconv}) of the width function; as $\Hopf_\xo\ein_\reg:\reg\to\ndpt\reg\times\RR \cong[\reg]$ is a topological quotient map by Lemma~\ref{weakimpliesstrong}, $[\zero] \subseteq [\reg]$ is also a Borel subset. 
Notice also that $\Hopf_\xo\ein_{\zero}:\zero\to\ndpt\zero\times\RR\cong [\zero]$ is a homeomorphism.  So if $\quotient{\Gamma}{[\zero]}$ 
has positive mass with respect to the weak Ricks' measure $\overline m_\Gamma$   we may define (as in \cite[Definition~8.12]{Ricks})   a geodesic flow and $\Gamma$-invariant measure $m^0$ on $\SX$ by setting 
\begin{equation}\label{defstrongRicks}
 m^0(E):= \overline m \bigl(\Hopf_\xo(E\cap \zero)\bigr)\quad\text{for any Borel set }\ E\subseteq\SX;
 \end{equation}
this measure $m^0$ then induces the {\hd Ricks' measure} $m^0_\Gamma$ on $\quotient{\Gamma}{ \SX}$. 

Notice that in general $\overline m_\Gamma (\quotient{\Gamma}{[\zero]})= 0\ $ is possible; obviously this is always the case when $\zero=\emptyset$. 
However, we will see later that under certain conditions the Ricks' measure is actually equal to the weak Ricks' measure used for its construction.

\section{The radial limit set and recurrence} \label{propradlimset}

As before  $\XX$ will always be a proper Hadamard space and $\Gamma<\is(\XX)$ a discrete rank one group. We further fix a  
base point $\xo\in\XX$. We will begin this section with a few definitions.

A point $\xi\in\rand$ is called a {\hd radial limit point} if there exists $c>0$ and  sequences $(\gamma_n)\subseteq\Gamma$ and $(t_n)\nearrow\infty$ such that
\begin{equation}\label{radlimpoint} d\bigl(\gamma_n \xo, \sigma_{\xo,\xi}(t_n)\bigr)\le c\quad\text{for all }\ n\in\NN.\end{equation}
Notice that by the triangle inequality this condition is independent of the choice of $\xo\in\XX$.
The {\hd radial limit set} $\radlim\subseteq\Lim$ of $\Gamma$ is defined as the set of radial limit points. 

Recall the notion of (weakly) $\Gamma$-recurrent  elements from Definition~\ref{weakstrongrecurrencedef}. 
 Moreover, an element 
 $v\in \SX$ is called  {\hd  $\Gamma$-divergent} if for every compact set $K\subseteq\SX$ there exists $T>0$ such that for all $t\ge T$ 
\[ g^tv\notin \bigcup_{\gamma\in\Gamma}\gamma K;\] 
it is called {\hd  weakly $\Gamma$-divergent} if for every compact set $\overline K\subset[\SX]$ there exists $T>0$ such that for all $t\ge T$ 
\[ g^t [v]\notin \bigcup_{\gamma\in\Gamma}\gamma \overline K.\] 

For the convenience of the reader we state the following easy fact. 
\begin{lemma}\label{critradlim}
Let $ u\in \SX$.  
Then 
\[  u\ \ \Gamma\text{-recurrent}\quad\Longrightarrow\quad u^+\in\radlim
 \quad\Longrightarrow\quad  u  \  \text{ \underline{not} }\ \Gamma\text{-divergent}.\]
\end{lemma}

We want to emphasize here that 
 in general  $u$  weakly $\Gamma$-recurrent does \underline{not} imply $u^+\in\radlim$, while $u$ not  $\Gamma$-divergent always implies $u$ not weakly $\Gamma$-divergent. However, if \underline{$u\in\reg$} is weakly $\Gamma$-recurrent, then according to Lemma~\ref{weakimpliesstrong}  $u$ $\Gamma$-accumulates to some $w\sim u$. This again implies that $w^+=u^+\in\radlim$ and we get the following 
\begin{lemma}\label{critregradlim}
If $ u\in \reg$ 
then 
\[ u\ \text{ weakly }\ \Gamma\text{-recurrent}\quad\Longrightarrow\quad u^+\in\radlim
 \quad\Longrightarrow\quad u  \  \text{ \underline{not} weakly }\ \Gamma\text{-divergent}.\]
 \end{lemma}
In the sequel the following subsets of $\SX$ will be convenient. Notice that for $v\in\SX$ the reverse geodesic $-v\in\SX$ is defined by $-v(s):=v(-s)$ for all $s\in\RR$.
\begin{align*}
\SX_{\Gamma}^{\small{\mathrm{rad}}}&:=\{v\in\SX\colon v^-\in\radlim, \ v^+\in\radlim\},\\
\SX_{\Gamma}^{\small{\mathrm{rec}}}&:=\{v\in\SX\colon v\ \text{and } -v\ \text{are }  \Gamma\text{-recurrent}\},\\
\SX_\Gamma^{\small{\mathrm{div}}}&:=\{v\in\SX\colon v\ \text{and } -v\ \text{are  } \Gamma\text{-divergent}\},\\
\SX_\Gamma^{\small{\mathrm{wrec}}}&:=\{v\in\SX\colon v\ \text{and } -v\ \text{are weakly } \Gamma\text{-recurrent}\},\\
\SX_\Gamma^{\small{\mathrm{wdiv}}}&:=\{v\in\SX\colon v\ \text{and } -v\ \text{are weakly } \Gamma\text{-divergent}\}.
\end{align*}

Notice that   in general $[\SX_\Gamma^{\small{\mathrm{wrec}}}]\subsetneq [\SX_\Gamma^{\small{\mathrm{rec}}}]$ and even 
\[ [\SX_\Gamma^{\small{\mathrm{wrec}}}\cap\reg ]\subsetneq [\SX_\Gamma^{\small{\mathrm{rec}}}\cap\reg]\]
by the remark following Definition~\ref{weakstrongrecurrencedef}.  

From now on we will also deal with the quotient $\quotient{\Gamma}{\SX}$; for the remainder of this section we will  therefore denote elements in the quotient by $u,v, w$ and elements in $\SX$ by $\widetilde u, \widetilde v,\widetilde w$.   
According to the definitions given in Section~\ref{dyndef}, $v\in \quotient{\Gamma}{\SX}$ is positively and negatively recurrent if and only if every lift $\widetilde v$ of $v$ belongs to $\SX_{\Gamma}^{\small{\mathrm{rec}}}$; $ v\in \quotient{\Gamma}{\SX}$ is positively and negatively divergent if and only if every lift $\widetilde v$ of $v$ belongs to $\SX_{\Gamma}^{\small{\mathrm{div}}}$. Similarly,  $[ v]\in \quotient{\Gamma}{ [\SX]}$  is positively and negatively recurrent if and only if for every lift $[\widetilde v]\in[\SX]$ and every representative $\widetilde u\in\SX$ of $[\widetilde v]$ we have $\widetilde u\in \SX_{\Gamma}^{\small{\mathrm{wrec}}}$; $[ v]\in \quotient{\Gamma}{[\SX]}$ is positively and negatively divergent if and only if for  every lift $[\widetilde v]\in[\SX]$ and every representative $\widetilde u\in\SX$ of $[\widetilde v]$ we have $\widetilde u\in \SX_{\Gamma}^{\small{\mathrm{wdiv}}}$. 
 
 We now assume that $m_\Gamma$ is a Knieper's measure on $\quotient{\Gamma}{\SX}$ constructed from an arbitrary geodesic current $\overline \mu$ and that $\overline m_\Gamma$ is a weak Ricks' measure on $\quotient{\Gamma}{[\SX]}$ coming from a geodesic current $\overline \mu\,$ defined on $\ndpt\reg$. 
For the convenience of the reader we state and prove the following easy
\begin{lemma}[compare also Theorem~2.3 in \cite{MR1293874}] \label{consdiss}\  \\
The dynamical systems $\bigl(\quotient{\Gamma}{ \SX}, g_\Gamma, m_\Gamma\bigr)$ respectively $\bigl(\quotient{\Gamma}{ [\SX]}, g_\Gamma, \overline m_\Gamma\bigr)$
are
\begin{enumerate}
\item[(a)] conservative  if and only if $\ \overline\mu\bigl(\ndpt (\SX\setminus \SX_{\Gamma}^{\small{\mathrm{rad}}})\bigr)=0$, 
\item[(b)] dissipative  if and only if $\ \overline\mu(\ndpt\SX_{\Gamma}^{\small{\mathrm{rad}}})=0$. 
\end{enumerate}
Moreover, in the dissipative case the measures $m_\Gamma$ and $\overline m_\Gamma$ are infinite, and the corresponding dynamical systems are non-ergodic unless $\overline\mu\,$ is  supported on a single 
orbit $\,\Gamma\cdot (\xi,\eta)\subseteq\ndpt\SX$.
\end{lemma}
\begin{proof}
We first treat the dynamical system $\bigl(\quotient{\Gamma}{ \SX}, g_\Gamma, m_\Gamma\bigr)$ with Knieper's measure $m_\Gamma$;  let $\Omega_D$ denote its dissipiative part and $\Omega_C$ its conservative part. 
Then by Poincar{\'e}'s recurrence theorem and Hopf's divergence theorem we have
\[  m_\Gamma (\Omega_D)= m_\Gamma\bigl(\quotient{\Gamma}{\SX_\Gamma^{\small{\mathrm{div}}}}\bigr)\quad\text{and }\  m_\Gamma (\Omega_C)= m_\Gamma\bigl(\quotient{\Gamma}{\SX_{\Gamma}^{\small{\mathrm{rec}}}}\bigr).\]
Moreover, Lemma~\ref{critradlim}  implies
\[\SX_\Gamma^{\small{\mathrm{div}}} \subseteq  \SX\setminus \SX_{\Gamma}^{\small{\mathrm{rad}}}\quad\text{and }\ \SX_\Gamma^{\small{\mathrm{rec}}}\subseteq \SX_{\Gamma}^{\small{\mathrm{rad}}},\]
and as $\SX= \SX_{\Gamma}^{\small{\mathrm{rad}}}\sqcup \SX\setminus \SX_{\Gamma}^{\small{\mathrm{rad}}}$ we get
\[  m_\Gamma (\Omega_D)= m_\Gamma\bigl(\quotient{\Gamma}{ (\SX\setminus \SX_{\Gamma}^{\small{\mathrm{rad}}})}\bigr)\quad\text{and }\  m_\Gamma (\Omega_C)= m_\Gamma\bigl(\quotient{\Gamma}{\SX_{\Gamma}^{\small{\mathrm{rad}}}}\bigr).\]

Hence by construction of Knieper's measure from the geodesic current $\overline\mu$, the dynamical system $\bigl(\quotient{\Gamma}{ \SX}, g_\Gamma, m_\Gamma\bigr)$ is conservative if and only if $\ \overline\mu\bigl(\ndpt (\SX\setminus \SX_{\Gamma}^{\small{\mathrm{rad}}})\bigr)=0$, and it is dissipative if and only if $\ \overline\mu(\ndpt \SX_\Gamma^{\small{\mathrm{rad}}})=0$.
 
We next treat  the dynamical system $\bigl(\quotient{\Gamma}{[ \SX]}, g_\Gamma, \overline m_\Gamma\bigr)$; let   $\overline\Omega_D$ denote  its dissipative part and $\overline\Omega_C$ its conservative part. Then again by Poincar{\'e}'s recurrence theorem and Hopf's divergence theorem  we have
\[  \overline m_\Gamma (\overline \Omega_D)= \overline m_\Gamma\bigl(\quotient{\Gamma}{[\SX_\Gamma^{\small{\mathrm{wdiv}}}]}\bigr)\quad\text{and }\  \overline m_\Gamma (\overline \Omega_C)=\overline  m_\Gamma\bigl(\quotient{\Gamma}{[\SX_{\Gamma}^{\small{\mathrm{wrec}}}]}\bigr).\] 
From Lemma~\ref{critregradlim} we further get  
  \[ [\reg\cap \SX_\Gamma^{\small{\mathrm{wdiv}}}]
    \subseteq  [\reg\cap \SX\setminus \SX_{\Gamma}^{\small{\mathrm{rad}}}]\quad\text{and }\ [\reg\cap \SX_\Gamma^{\small{\mathrm{wrec}}}]\subseteq [\reg\cap\SX_{\Gamma}^{\small{\mathrm{rad}}}].\]
Since   $[\reg]= [ \reg\cap \SX_{\Gamma}^{\small{\mathrm{rad}}}]\sqcup[\reg\cap \SX\setminus \SX_{\Gamma}^{\small{\mathrm{rad}}}]$ and as the weak Ricks' measure is supported on $\quotient{\Gamma}{[\reg]}$,  we conclude
\[  \overline m_\Gamma (\overline \Omega_D)= \overline m_\Gamma\bigl(\quotient{\Gamma}{ [\SX\setminus \SX_{\Gamma}^{\small{\mathrm{rad}}}]}\bigr)\quad\text{and }\  \overline m_\Gamma (\overline \Omega_C)= \overline m_\Gamma\bigl(\quotient{\Gamma}{[\SX_{\Gamma}^{\small{\mathrm{rad}}}}]\bigr).\]

So by construction of the weak Ricks' measure from the geodesic current $\overline\mu\,$ defined on $\ndpt\reg$, the dynamical system $\bigl(\quotient{\Gamma}{ [\SX]}, g_\Gamma,  \overline m_\Gamma\bigr)$ is conservative if and only if $\ \overline\mu\bigl(\ndpt [\SX\setminus \SX_{\Gamma}^{\small{\mathrm{rad}}}]\bigr)=\overline\mu\bigl(\ndpt (\SX\setminus \SX_{\Gamma}^{\small{\mathrm{rad}}})\bigr)=0$, and it is dissipative if and only if $\ \overline\mu\bigl(\ndpt [\SX_\Gamma^{\small{\mathrm{rad}}}]\bigr)=\overline\mu(\ndpt \SX_\Gamma^{\small{\mathrm{rad}}})=0$.

The last statement is obvious (see the  paragraph before Theorem~\ref{Hopfindividual}).
\end{proof}
As a consequence we get the following statement which generalizes Lemma~7.5 in \cite{Ricks} (where the stronger assumption of a {\hl finite} weak Ricks' measure $\overline m_\Gamma$ is needed): 
\begin{corollary}\label{weakregisfull}
Let $\overline\mu\,$ be a geodesic current defined on $\ndpt\reg$. Then 
\[ \overline\mu\bigl(\ndpt(\SX\setminus \SX_\Gamma^{\small{\mathrm{rad}}})\bigr)=0\quad\Longrightarrow\quad  \overline \mu\bigl(\ndpt (\SX\setminus \SX_\Gamma^{\small{\mathrm{wrec}}})\bigr)=0.\]
\end{corollary}
\begin{proof}
For the weak Ricks' measure $\overline m_\Gamma$ associated to the geodesic current   $\overline\mu\,$ the conservative part $\overline\Omega_C$ satisfies
\[\overline m_\Gamma(\overline\Omega_C) =  \overline m_\Gamma(\quotient{\Gamma}{[\SX]})\]
according to Lemma~\ref{consdiss} (b); from the proof above we further have
\[ \overline m_\Gamma(\overline\Omega_C)=\overline m_\Gamma\bigl(\quotient{\Gamma}{[\SX_{\Gamma}^{\small{\mathrm{wrec}}}}]\bigr).\]  
Hence by construction of the weak Ricks' measure we conclude
\[ \overline \mu\bigl(\ndpt (\SX\setminus \SX_\Gamma^{\small{\mathrm{wrec}}})\bigr)=\overline \mu\bigl(\ndpt [\SX\setminus \SX_\Gamma^{\small{\mathrm{wrec}}}]\bigr)=0.\]
\end{proof}
 
In the sequel we will use this result to prove the necessary generalizations of Corollary~8.3, Lemma~8.5 and Lemma~8.6 in \cite{Ricks}, which were  only proved for geodesic currents coming from a conformal density as defined in (\ref{overlinemudef}),  and which induce a {\hl finite} Ricks' measure. 

For the remainder of this section  we fix non-atomic probability measures $\mu_-$, $\mu_+$ on $\rand$
with $\supp(\mu_{\pm})=\Lim$, and let  
\[\overline\mu\sim 
(\mu_-\otimes\mu_+)\ein_{\ndpt\reg}\]
be a quasi-product geodesic current defined on $\ndpt\reg$. 

Notice that since the support of $\mu_-$ and $\mu_+$ equals $\Lim$, minimality of the limit set  $\Lim$  (see for example \cite[Proposition~2.8]{MR656659}) implies that every open subset $U\subseteq\rand$ with $U\cap\Lim\ne\emptyset$ satisfies $\mu_{\pm}(U)>0$.   Hence if $h\in\Gamma$ is a rank one element, then for the open neighborhoods $U^-$, $U^+\subseteq\ganz$ of $h^-$, $h^+$  provided by Lemma~\ref{joinrankone} we know that 
\begin{equation}\label{overlinemunotzero}
 (\mu_-\otimes \mu_+)(\ndpt\reg)\ge  \mu_-(U^-)\cdot \mu_+(U^+)>0;\end{equation}
so $\overline\mu$ is non-trivial. Moreover, 
 according to the Main Theorem in \cite{MR2581914} (see also Proposition~6.6 (3) in \cite{Ricks}), the set $\ndpt\reg\cap(\Lim\times\Lim)$ is dense in $\Lim\times\Lim$, hence   \[ \supp(\overline\mu) = \Lim\times \Lim.\]
 
 The first Lemma shows that  in the setting of Lemma~\ref{consdiss} (a) -- that is when the weak Ricks' measure 
associated to $\overline\mu\,$ is conservative, but not necessarily finite  -- we have $\overline\mu\sim \mu_-\otimes \mu_+$; in other words  we may omit the restriction to $\ndpt\reg$. 

\begin{lemma}[Corollary~8.3 in \cite{Ricks}]\label{regnotnecessary}
If  
$\ \overline\mu\bigl(\ndpt(\SX\setminus \SX_\Gamma^{\small{\mathrm{rad}}})\bigr)=0,$ 
then
 \[(\mu_-\otimes\mu_+)(\ndpt\reg)= (\mu_-\otimes\mu_+)(\rand\times\rand)=1.\]
 \end{lemma}
\begin{proof}
From the hypothesis and  Corollary~\ref{weakregisfull}  we get $\   \overline\mu\bigl(\ndpt(\SX\setminus  \SX_\Gamma^{\small{\mathrm{wrec}}})\bigr)=0$ and hence 
 \begin{equation}\label{weakreczero} (\mu_-\otimes\mu_+) \bigl(\ndpt(\reg\setminus  \SX_\Gamma^{\small{\mathrm{wrec}}})\bigr)=0.\end{equation}
In a first step we prove that the set 
\[A:=\{\xi\in\rand\colon (\xi,\eta)\in\ndpt\reg \quad\text{for all }\ \eta\in\rand,\ \eta\ne \xi  \}\]
satisfies $\mu_-(A)=\mu_+(A)=1$.   So let $\xi\in\rand$ be arbitrary. Our goal is to show that $\xi$ possesses an open neighborhood $U\subseteq\rand$ with $\mu_-(U\setminus A)=0$; the claim then follows by compactness of $\rand$ (and analogously for $\mu_+$ instead  of $\mu_-$).

Let $h\in\Gamma\,$ be a rank one element. According to Lemma~\ref{dynrankone} (a)   there exists $w\in\reg$ with $w^-=\xi$ and $w^+=h^+$.  Lemma~\ref{joinrankone} then provides open neighborhoods $U$, $V\subseteq\rand$ of $\xi$, $h^+$  \st $U\times V\subseteq\ndpt\reg$. 
From (\ref{weakreczero}) we get $(\mu_-\otimes\mu_+) \bigl((U\times V)\setminus  \SX_\Gamma^{\small{\mathrm{wrec}}}\bigr)=0$.

For the subset \[ W=\{\zeta\in U\colon \exists\, u\in\SX_\Gamma^{\small{\mathrm{wrec}}}\ \ \st \ \  u^-=\zeta,\ u^+\in V\}\subseteq U\]
of $U$ we have the inclusion 
 $(U\setminus W)\times V\subseteq (U\times V)\setminus \SX_\Gamma^{\small{\mathrm{wrec}}}$. Hence
\[ 0= (\mu_-\otimes\mu_+) \bigl((U\times V)\setminus  \SX_\Gamma^{\small{\mathrm{wrec}}}\bigr)\ge \mu_-(U\setminus W)\cdot\mu_+(V),\]
and from $\mu_+(V)>0$ we get $\mu_-(U\setminus W)=0$. 
As Lemma~\ref{jointoweakrecurrent} implies  $W\subseteq A$, we conclude $ \mu_-(U\setminus A)\le \mu_-(U\setminus W)=0$.

Finally we let $\xi\in A$ arbitrary. So  for all $\eta\in\rand\setminus\{\xi\}$ we have $(\xi,\eta)\in\ndpt\reg$. Since $\mu_+(\{\xi\})=0$ by non-atomicity of $\mu_+$, we have $(\xi,\eta)\in\ndpt\reg$ for $\mu_+$-almost every $\eta\in\rand$. The claim then follows from $\mu_-(A)=1$ and Fubini's Theorem.
\end{proof}
From the previous lemma and the proof of Lemma~\ref{consdiss} we immediately get
\begin{corollary}
 $\ \overline\mu\bigl(\ndpt(\SX\setminus \SX_\Gamma^{\small{\mathrm{rad}}})\bigr)=0\, $ if and only if 
$\mu_-(\radlim)=\mu_+(\radlim)=1$. \end{corollary}
For the remainder of this section we use the previous assumptions on $\mu_-$, $\mu_+$ and $\overline\mu$; moreover we will require that 
\[ \mu_-(\radlim)=\mu_+(\radlim)=1.\]

\begin{lemma}[Lemma~8.5 in \cite{Ricks}]\label{mapisconstantae}
Let $S$ be any set and $\Psi:\ndpt\reg\to S$ an arbitrary map. If $\Omega\subseteq\ndpt\reg$ is a set of full $\overline\mu$-measure in $\ndpt\reg$ such that for all $(\xi,\eta)$, $(\xi,\eta')$, $(\xi',\eta')\in\Omega$ we have
\[ \Psi\bigl((\xi,\eta)\bigr)=\Psi\bigl((\xi,\eta')\bigr)=\Psi\bigl((\xi',\eta')\bigr),\]
then $\Psi$ is constant $\overline\mu$-almost everywhere on $\ndpt\reg$.
\end{lemma}
\begin{proof} 
From Lemma~\ref{regnotnecessary} and $\overline\mu(\ndpt\reg\setminus\Omega)=0$  we get 
\[ (\mu_-\otimes\mu_+)(\Omega)=(\mu_-\otimes\mu_+)(\ndpt\reg)=(\mu_-\otimes\mu_+)(\rand\times\rand).\]  
  Hence for $\mu_-$-almost every $\xi\in\rand$ the set 
\[ B_\xi:=\{\eta\in\rand\colon (\xi,\eta)\in\Omega\}\]
has full $\mu_+$-measure in $\rand$; in particular, the set 
\[ A:=\{\xi\in\rand\colon \mu_+(B_\xi)=\mu_+(\rand)=1\}\]
satisfies $\mu_-(A)=\mu_-(\rand)=1$. 

We now fix $(\xi,\eta)\in (A\times\rand)\cap\Omega$. Then for any $(\xi',\eta')\in (A\times B_\xi)\cap\Omega$ we have $(\xi,\eta')\in (A\times B_\xi)\cap \Omega$,  hence by hypothesis on $\Omega$ 
\[ \Psi\bigl((\xi',\eta')\bigr)=\Psi\bigl((\xi,\eta')\bigr)=\Psi\bigl((\xi,\eta)\bigr).\]
Since the set $(A\times B_\xi)\cap\Omega\subseteq\ndpt\reg$ has full $(\mu_-\otimes\mu_+)$-measure in $\rand\times\rand$, it also has full $\overline\mu$-measure in $\ndpt\reg$. So  we get  $\Psi\bigl((\xi',\eta')\bigr)=\Psi\bigl((\xi,\eta)\bigr)$ for $\overline\mu$-almost every $(\xi',\eta')\in\ndpt\reg$, and hence $\Psi$ is constant $\overline\mu$-almost everywhere on $\ndpt\reg$.
\end{proof}

The following lemma together with Lemma~\ref{Hausdorffconv} is the clue to the proof of Theorem~\ref{zerofull}. 
\begin{lemma}[Lemma~8.6 in \cite{Ricks}]\label{isometrytypeconstant}
For $\overline\mu$-almost every $(\xi,\eta)\in\ndpt\reg$ the isometry type of $C_{(\xi\eta)}$ is the same.
\end{lemma}
\begin{proof}
According to Corollary~\ref{weakregisfull} the set $\ndpt (\SX_\Gamma^{\small{\mathrm{wrec}}}\cap\reg)$ has full $\overline\mu$-measure in $\ndpt\reg$. 
Moreover, if $u,v\in \SX_\Gamma^{\small{\mathrm{wrec}}}\cap\reg$ satisfy $u^-=v^-$ or $u^+=v^+$, then by Lemma~\ref{weakgivesisometricembeddings} there exist isometric embeddings between the compact metric spaces $C_u$ and $C_v$; hence $C_u$ and $C_v$ are isometric according to Theorem~1.6.14 in \cite{MR1835418}. The claim now follows by applying Lemma~\ref{mapisconstantae} to the map which sends $(\xi,\eta)\in \ndpt (\SX_\Gamma^{\small{\mathrm{wrec}}}\cap\reg)$ to the isometry type of $C_{(\xi\eta)}$. 
\end{proof}

We will now prove the appropriate generalization of Theorem~8.8 in \cite{Ricks}, which states  that  under the additional hypothesis  $\zero_\Gamma\ne \emptyset$  -- which is satisfied in particular if $\XX$ is geodesically complete --  the set $\ndpt\zero$ of end-point pairs of zero width geodesics has full $(\mu_-\otimes\mu_+)$-measure in $\rand\times\rand$. This will provide the key in the proof of ergodicity in Section~\ref{HopfArgument}. Moreover, it implies that any weak Ricks' measure $\overline m_\Gamma$ on $\quotient{\Gamma}{ [\SX]}$ associated to  a  quasi-product geodesic current $\overline\mu\sim(\mu_-\otimes\mu_+)\ein_{\ndpt\reg} $ is equivalent to the induced Ricks' measure $m^0_\Gamma$ on $\quotient{\Gamma}{\SX}$.  
\begin{theorem}\label{zerofull}
Let $\XX$ be a proper Hadamard space and $\Gamma<\is(\XX)$ a discrete rank one group \st  $\zero_\Gamma\ne \emptyset$. If $\mu_-$, $\mu_+$ are non-atomic probability measures  on $\rand$ with $\supp(\mu_{\pm})=\Lim$ and $\mu_-(\radlim)=\mu_+(\radlim)= 1$, 
then 
\[(\mu_-\otimes\mu_+)(\ndpt \zero)=1.\]
Moreover, if $\,\overline\mu$ is a quasi-product geodesic current absolutely continuous with respect to $(\mu_-\otimes\mu_+)\ein_{\ndpt\reg} $, then
\[ \overline\mu\bigl(\ndpt(\SX\setminus \zero)\bigr)=0.\]
\end{theorem}
\begin{proof}
By Lemma~\ref{isometrytypeconstant} there exists a set 
$\Omega \subseteq \ndpt\reg$ of full $\overline\mu$-measure in $\ndpt\reg$ such that the isometry type of  $C_{(\xi\eta)}$ is the same for all $(\xi,\eta)\in\Omega$.   Lemma~\ref{regnotnecessary} then implies
 \begin{equation}\label{Omegafull}  (\mu_-\otimes\mu_+)(\Omega)=1.\end{equation}

Fix $v\in\zero_\Gamma$ and let $U^-$, $U^+\subseteq\ganz$ be open neighborhoods of $v^-$, $v^+$ according to Lemma~\ref{joinrankone}. Consider decreasing sequences of open subsets $(U_n^-)\subseteq U^-\cap\rand$, $(U_n^+)\subseteq U^+\cap\rand$ \st 
 \[\bigcap_{n\in\NN} U_n^-=\{v^-\}\quad\text{and }\   \bigcap_{n\in\NN} U_n^+=\{v^+\}.\] 
Let $n\in\NN$. As 
$\supp(\mu_{\pm})=\Lim$, we get $(\mu_-\otimes\mu_+)(U_n^-
\times U_n^+)=\mu_-(U_n^-)\cdot\mu_+(U_n^+)
>0$, hence by   (\ref{Omegafull}) 
 \[(\mu_-\otimes\mu_+)\bigl(\Omega\cap (U_n^-
\times U_n^+)\bigr)>0.\] 
   So in particular there exists $(\xi_n,\eta_n)\in (U_n^-\times U_n^+)\cap\Omega$. 
 
 By choice of the sets $U_n^-$, $U_n^+$ we get a sequence $\bigl((\xi_n,\eta_n)\bigr)\subseteq \Omega\subseteq\ndpt\reg$ which converges to $(v^-,v^+)\in\ndpt\zero_\Gamma$.   Now Lemma~\ref{Hausdorffonboundary}  implies that some subsequence of $\bigl(C_{(\xi_n\eta_n)}\bigr)$ converges, in the Hausdorff metric, to a point. As the isometry type of $C_{(\xi\eta)}$ is the same for all  $(\xi,\eta)\in\Omega$, this implies that $C_{(\xi\eta)}$ is a point for all  $(\xi,\eta)\in\Omega$, hence $ \Omega\subseteq\ndpt\zero$. We conclude
 \[ (\mu_-\otimes\mu_+)(\ndpt\zero)\ge (\mu_-\otimes\mu_+)(\Omega)=1,\]
hence
 $ \overline\mu\bigl(\ndpt(\SX\setminus\zero)\bigr) =0$.
 \end{proof}
\begin{corollary}\label{weakisstrongisKnieper}
Let $\XX$ be a proper Hadamard space and $\Gamma<\is(\XX)$ a discrete rank one group \st $\zero_\Gamma\ne \emptyset$.   Let $\mu_-$, $\mu_+$ be non-atomic probability measures  on $\rand$ with $\supp(\mu_{\pm})=\Lim$
 and $\mu_-(\radlim)=\mu_+(\radlim)= 1$,
and $\,\overline\mu\sim (\mu_-\otimes\mu_+)\ein_{\ndpt\reg} $  a quasi-product geodesic current defined on $\ndpt\reg$. 
Then the weak Ricks' measure associated to $\overline\mu$  is 
equal to the Ricks' measure defined by (\ref{defstrongRicks}) and  also to any Knieper's measure associated to the quasi-product geodesic current $\overline\mu$ (if it exists).
\end{corollary}

\section{Conservativity versus ergodicity}\label{HopfArgument}

As before let $\XX$ be a proper Hadamard space with fixed base point $\xo\in\XX$. For $R>0$ we denote ${\mathcal B}(R)\subseteq \SX$ the set of all parametrized geodesic lines with origin in $B_\xo(R)$. 

In this section we assume that $\Gamma<\is(\XX)$ is a discrete rank one group with
 \[\zero_\Gamma:=\{v\in\zero \colon v^+, v^-\in\Lim\}\ne \emptyset.\]  
 Notice that if $\XX$ is geodesically complete, then according to Proposition~\ref{largewidthgiveszerowidth}  the latter condition is automatically satisfied. 

Throughout the whole section we fix  non-atomic probability measures  $\mu_-$, $\mu_+$ on $\rand$ with $\supp(\mu_{\pm})=\Lim$ and $\mu_-(\radlim)=\mu_+(\radlim)= 1$. 
Let $\overline\mu\sim (\mu_-\otimes\mu_+)\ein_{\ndpt\reg}$ be  a quasi-product geodesic current defined on $\ndpt\reg$ 
for which
 \begin{equation}\label{boundongrowth} \Delta:=\sup \Big\{ \frac{\ln  \overline\mu\bigl(\ndpt {\mathcal B}(R)\bigr)}{R}\colon R>0\Big\} 
\end{equation}
 is finite.

We next consider Ricks' measure $m_\Gamma^0$ associated to the geodesic current $\overline\mu\,$ as defined in~(\ref{defstrongRicks}). Since in the given setting Corollary~\ref{weakisstrongisKnieper} implies that Ricks' measure is 
equal to weak Ricks' measure and also to Knieper's measure associated to the same geodesic current $\overline\mu$, we will denote Ricks' measure by $m_\Gamma$ instead of $m_\Gamma^0$. Notice that  by assumption on $\mu_-$ and $\mu_+$ the set $\SX_\Gamma^{\small{\mathrm{rad}}} $ has full $\overline\mu$-measure; so  we already know from Lemma~\ref{consdiss} that $(\quotient{\Gamma}{\SX}, g_\Gamma, m_\Gamma)$ is conservative. 
 The goal of this section is to prove that it is also ergodic.
 
The proof of ergodicity will make   use of the famous  Hopf argument (see \cite{Hopf}, \cite{MR0284564}) as in \cite{MR2057305} and \cite{LinkPicaud}, for which Theorem~\ref{zerofull} is indispensable. 
In our more general setting including singular spaces we first need an analogon to Knieper's Proposition~4.1 which is valid only for manifolds. 
We remark that in view of Lemma~\ref{jointoweakrecurrent} our generalization of Knieper's Proposition~4.1  is not very surprising.
\begin{lemma}\label{KniepersProp}\ 
\  Let $u\in {\mathcal Z}$ be a $\Gamma$-recurrent rank one geodesic of zero width. Then for all $v\in \SX$ with $v^+=u^+$ and $\bs_{v^+}(v(0),u(0))=0$ we have
$$ \lim_{t\to\infty} d_1(g^t v, g^tu)=0.$$
\end{lemma}
\begin{proof} Since $u$ is $\Gamma$-recurrent, there exist sequences $(\gamma_n)\subseteq\Gamma$ and $(t_n)\nearrow\infty$ \st $\gamma_n g^{t_n}u$ converges to $u$. Let $v\in \SX$ be a geodesic such that  $v^+=u^+$ and $\bs_{v^+}(v(0),u(0))=0$. Then the function
\[ [0,\infty)\to [0,\infty),\quad t\mapsto d_1(g^t v,g^t u)=\sup\{ \e^{-|s|}d(v(t+s),u(t+s))\colon  s\in\RR\}\] is monotone decreasing as the geodesic rays determined by $u$ and $v$ are asymptotic. If the function does not converge to zero as $t$ tends to infinity, there exists a constant $\epsilon>0$ \st 
\[ d_1(g^t v,g^t u)\ge \epsilon \] for all $t\ge 0$ and hence
\[ \epsilon \le d_1 (g^{t_n+s} v,g^{t_n+s} u)\le d_1(v,u)\]
for all $s\ge -t_n$. By $\Gamma$-invariance of $d_1$ we get for all $n\in\NN$ and for all $s\ge -t_n$
\[ \epsilon \le d_1 (g^s \gamma_n g^{t_n} v,g^s\gamma_n g^{t_n} u)\le d_1(v,u).\]
Passing to a subsequence if necessary we may assume that $\gamma_n g^{t_n} v$ converges to some $\overline{v}\in \SX$. Hence in the limit as $n\to\infty$ we get
\[ \epsilon \le d_1 (g^s \overline{v},g^s u)\le d_1(v,u)\le \max\{2, d( v(0),u(0))\}\]
for all $s\in\RR$. Now the first inequality shows that $\overline{v}\ne u$ and the  second inequality gives $(\overline{v}^-,\overline{v}^+)=(u^-,u^+)$, which means that the geodesic lines $\overline{v}$ and $u$ are parallel. Notice that in this case $H_\xo(\overline{v})=H_\xo(u)$  if and only if $\bs_{u^-}\bigl(\overline{v}(0),u(0)\bigr)=0$ if and only if $\bs_{u^+}\bigl(\overline{v}(0),u(0)\bigr)=0$.    By choice of $v$ we have 
for all $n\in\NN$
\begin{align*}
 0&=\bs_{u^+}(v(t_n),u(t_n))
 = \lim_{s\to\infty}\bigl( d(v(t_n), u(t_n+s))-d(u(t_n),u(t_n+s)\bigr)\\
& = \lim_{s\to\infty}\bigl( d(\gamma_n v(t_n), \gamma_n u(t_n+s))-s\bigr)= \lim_{s\to\infty}\bigl( d\bigl((\gamma_n g^{t_n} v)(0), (\gamma_n g^{t_n} u)(s)\bigr)-s\bigr);
 \end{align*}  
by definition of $\overline{v}$ and $\Gamma$-recurrence of $u$ this gives  
\begin{align*}
 0&= \lim_{s\to\infty}\bigl( d( \overline{v}(0), u(s))-s\bigr) =\bs_{u^+}(\overline{v}(0),u(0)).
  \end{align*}  
Hence   $\overline{v}\sim u$ which is a  contradiction to  $\overline{v}\ne u$ and $u\in\zero$. 
\end{proof} 
Since we want to apply Hopf's criterion for ergodicity Theorem~\ref{Hopfindividual} we need to find an appropriate function $\rho:\quotient{\Gamma}{\SX}\to \RR$ in $\LL^1(m_\Gamma)$ which is strictly positive $m_\Gamma$-almost everywhere. Let $\Delta\ge 0$ be the constant defined by (\ref{boundongrowth}).  
\begin{lemma}\label{definerho}
 The function 
\[\widetilde \rho:\SX\to\RR,\quad u\mapsto\Biggl\{\begin{array}{cl} \displaystyle \max\{ \e^{-2\Delta  d(u(0),\gamma\xo)}\colon \gamma\in\Gamma\} & \text{if } \ u\in \zero\\[1mm]
0 & \text{if } \ u\in \SX\setminus\zero
\end{array}\]
descends to a function $\rho: \quotient{\Gamma}{\SX}\to\RR$ which is strictly positive $m_\Gamma$-almost everywhere and belongs to $\LL^1(m_\Gamma)$.
Moreover, if $u,v\in\zero$ satisfy $d\bigl(u(0),v(0)\bigr)\le 1$, then
\[ |\widetilde\rho(u)-\widetilde\rho(v)|\le \widetilde\rho(u)\cdot 2\Delta  \e^{2\Delta} d\bigl(u(0),v(0)\bigr).\]
\end{lemma} 
\begin{proof}
We first notice that by definition $\widetilde \rho$ is $\Gamma$-invariant and strictly positive on $\zero$, hence $\rho$ is well-defined and strictly positive $m_\Gamma$-almost everywhere (as $m_\Gamma(\quotient{\Gamma}{ \zero})= m_\Gamma(\quotient{\Gamma}{ \SX})$ by construction of Ricks' measure).  
By definition~(\ref{boundongrowth}) of $\Delta$ we get
\[ m\bigl( {\mathcal B}(R)\bigr) \le 2R\cdot \overline\mu\big( \ndpt {\mathcal B}(R)\bigr) \le 2R \e^{\Delta R}. \]
Let  ${\mathcal D}_\Gamma\subseteq\SX $ denote the Dirichlet domain for $\Gamma$ with center $\xo$, that is the set of all parametrized geodesic lines with origin in 
\[ \{x\in\XX\colon d(x,\xo)\le d(x,\gamma\xo)\quad\text{for all }\ \gamma\in\Gamma\};\]
then for all $u\in {\mathcal D}_\Gamma\cap\zero$ we have
\[ \widetilde\rho(u)= \e^{-2\Delta  d(u(0),\xo)}.\]
Notice that if $u \in {\mathcal S}(R):=\bigl({\mathcal B(R)\setminus {\mathcal B}(R-1)\bigr)\cap {\mathcal D}_\Gamma}\cap\zero $, then $d(u(0),\xo)\ge R-1$ and  we estimate 
\begin{align*}
 \int_{{\mathcal S}(R)} \widetilde\rho(u)\d m(u) & \le  \e^{-2\Delta  (R-1)}   \int_{{\mathcal B}(R)} \d m(u) \le 2R\e^{2\Delta}  \e^{-\Delta R} ;\end{align*}
this shows that $\rho\in \LL^1(m_\Gamma)$. 

We finally let $u,v\in \zero$ arbitrary with $d\bigl(u(0),v(0)\bigr)\le 1$. Let $\gamma,\gamma'\in\Gamma$ \st  $\widetilde\rho(u) =\e^{-2\Delta d(u(0),\gamma\xo)}$,  $\widetilde\rho(v) =\e^{-2\Delta d(v(0),\gamma'\xo)}$. Then
\begin{eqnarray*}
\widetilde\rho(u)-\widetilde\rho(v) &\le & \e^{-2\Delta  d(u(0),\gamma\xo)}\bigl(1- \e^{-2\Delta  d(u(0),v(0))}\bigr) ,\\
\widetilde\rho(v)-\widetilde\rho(u) &\ge &  \e^{-2\Delta  d(u(0),\gamma\xo)}\bigl( \e^{2\Delta  d(u(0),v(0))}-1\bigr),
\end{eqnarray*}
hence
\begin{align*} |\widetilde\rho(u)-\widetilde\rho(v)| &\le  \widetilde\rho(u)  \cdot \max\{1- \e^{-2\Delta  d(u(0),v(0))},   \e^{2\Delta  d(u(0),v(0))}-1\}\\
& \le \widetilde\rho(u) 2\Delta  \e^{2\Delta } d(u(0),v(0)).\end{align*}
\end{proof}
For the remainder of this section we will again denote elements in the quotient $\quotient{\Gamma}{\SX}$ be $u,v,w$ and elements in $\SX$ by $\widetilde u, \widetilde v, \widetilde w$.  
As we want to apply 
Theorem~\ref{Hopfindividual}, we state the following auxiliary result. Its proof is a straightforward computation as performed in \cite[page~144]{MR1041575} using the property of $\widetilde \rho\,$ stated in the last line of Lemma~\ref{definerho}. 
\begin{lemma}\label{pluslimitequal}
Let $f\in\Cnt_c(\quotient{\Gamma}{\SX})$ be arbitrary. If $u,v\in\quotient{\Gamma}{\zero}$ are positively recurrent with lifts $\widetilde u$, $\widetilde v$ satisfying $\widetilde u^+=\widetilde v^+$,  $\bs_{\widetilde v^+}(\widetilde u(0),\widetilde v(0))=0$ and such that 
\[  f^+(u):= \lim_{T\to\infty}  \frac{\int_0^T   f(g_\Gamma^{t} u)\d t}{\int_0^T  \rho(g_\Gamma^{ t}  u)\d t}\quad\text{and }\   f^+( v)=\lim_{T\to \infty}  \frac{\int_0^T   f(g_\Gamma^{t} v)\d t}{\int_0^T  \rho(g_\Gamma^{t} v)\d t}\] 
exist, then $ f^+( u)= f^+(v)$.
\end{lemma}

\begin{proposition}\label{conservativeimpliesergodic}\
The dynamical system $(\quotient{\Gamma}{ \SX}, (g^t_\Gamma)_{t\in\RR}, m_\Gamma)$ is   ergodic.
\end{proposition}
\begin{proof}  
Using the last statement of 
Theorem~\ref{Hopfindividual} we have to  show that  for every function $f\in \LL^1(m_\Gamma)$ the associated limit function $f^+$ defined by 
\[ f^+( u):= \lim_{T\to\infty} \frac{\int_0^T  f(g_\Gamma^{t}u)\d t}{\int_0^T \rho(g_\Gamma^{ t} u)\d t}\quad\text{for }\ m_\Gamma\text{-almost every } \  u\in \quotient{\Gamma}{\SX} \]
is constant $m_\Gamma$-almost everywhere; 
here  $\rho\in \LL^1(m_\Gamma)$ is the function defined in Lemma~\ref{definerho}. As $\Cnt_c(\quotient{\Gamma}{\SX})$ is dense in  $\LL^1(m_\Gamma)$ it will suffice to prove the claim for $f\in\Cnt_c(\quotient{\Gamma}{\SX})$.

So we choose $f\in\Cnt_c(\quotient{\Gamma}{\SX})$ arbitrary. Since $(\quotient{\Gamma}{\SX}, g_\Gamma, m_\Gamma)$ is conservative, 
Theorem~\ref{Hopfindividual} states that for $m_\Gamma$-almost every $ u\in \quotient{\Gamma}{\SX}$ the limits
$$   f^\pm(u)=\lim_{T\to +\infty} \frac{\int_0^T  f(g^{\pm t}_\Gamma u)\d t}{\int_0^T  \rho(g_\Gamma^{\pm t}u)\d t}$$
exist and are equal.

As $m_\Gamma$ is conservative and supported on $\quotient{\Gamma}{\zero}$, the set of recurrent elements in $\quotient{\Gamma}{\zero}$ has full measure in $\quotient{\Gamma}{\SX}$ with respect to $m_\Gamma$. So altogether the set
\begin{align*}
 \Omega &:=\{ u\in\quotient{\Gamma}{\zero} \colon  u\ \text{is positively and negatively recurrent}, \\
&\hspace*{2.5cm} f^{+}(u),\, f^-(u)\  \text{ exist and } \  f^+( u)= f^-( u)\} \end{align*}
has full measure in $\quotient{\Gamma}{\SX}$.

Moreover, from the local product structure of $m$  and Lemma~\ref{regnotnecessary} 
we know that there exists a lift $\widetilde w\in\SX$ of some  $w\in \Omega$  \st $$ G_{\widetilde w^-}:=\{\eta\in \rand \colon  \exists\, u\in\Omega \ \mbox{ with a lift }\ \widetilde u\in\zero  \ \mbox{ satisfying }\ \widetilde u^-=\widetilde w^-,\ \widetilde u^+=\eta\}$$
has full measure in $\rand$ \wrt $\mu_+$. This implies in particular that 
\begin{equation}\label{fullinxi} m_\Gamma\bigl(\{ v\in \Omega\colon \exists \ \text{lift } \ \widetilde v\in\zero \ \text{ satisfying } \ \widetilde v^+\in G_{\widetilde w^-}\}\bigr)=m_\Gamma(\quotient{\Gamma}{\SX}).\end{equation} 

We will next show that 
$ f^+$ is constant $m_\Gamma$-almost everywhere  on $\quotient{\Gamma}{\SX}$; according to~(\ref{fullinxi})  above  it suffices to show that for every $v\in\Omega$ with a lift   $\widetilde v\in \zero$ satisfying $\widetilde v^+\in G_{\widetilde w^-}$ we have  $ f^+(v)= f^+( w)$. So let $v\in\Omega$ be arbitrary with a lift $\widetilde v\in\zero$ satisfying $\widetilde v^+\in G_{\widetilde w^-}$. By definition of $G_{\widetilde w^-}$ there exists $u\in\Omega$ with a lift $\widetilde u\in \zero$ satisfying  $\widetilde u^-=\widetilde w^-$ and  $\widetilde u^+=\widetilde v^+$; replacing $\widetilde u$ by $g^s \widetilde u$ for an appropriate $s\in\RR$ if necessary we may further assume that $\bs_{\widetilde v^+}(\widetilde u(0),\widetilde v(0))=0$. Then the choice of $\widetilde w$,   the definition of $\Omega$ and Lemma~\ref{pluslimitequal} directly imply
\[  f^+(v)=f^+(u)=f^-( u).\] 
We next  choose  $s\in\RR$ such that $\bs_{\widetilde w^-}(\widetilde w(0),\widetilde u(s))=0$;  from the fact that $u$ is negatively recurrent, $\widetilde u^-=\widetilde w^-$ and Lemma~\ref{pluslimitequal}  we then get 
\[  f^-(w)=f^-(g^s_\Gamma  u).\]
As $f^\pm$ are $(g_\Gamma^t)$-invariant and $ w\in\Omega$, we conclude 
\[  f^+( v)=f^-( u) =f^-(g^s_\Gamma  u)=f^-( w)=f^+( w).\]
So we have shown that $m_\Gamma$-almost every $ v\in\quotient{\Gamma}{\SX}$ satisfies $ f^+( v)=f^+( w)$. 
\end{proof}

We now summarize the previous results to obtain 
\begin{theorem}\label{conservativestatement}\
Let $\Gamma<\is(\XX)$ be a discrete rank one group with $\zero_\Gamma\ne\emptyset$.   Let $\mu_-$, $\mu_+$ be non-atomic probability measures  on $\rand$ with $\supp(\mu_{\pm})=\Lim$, and 
$\overline\mu\sim  (\mu_-\otimes \mu_+)\ein_{\ndpt\reg}$  a quasi-product geodesic current on $\ndpt\reg$ for which the constant $\Delta\ge 0$ defined by (\ref{boundongrowth}) is finite. 

Let $m_\Gamma$ be the associated Ricks' measure  on $\quotient{\Gamma}{ \SX}$. Then 
 the following statements   are equivalent: \\[-3mm]
\begin{enumerate}
\item[(i)] $\mu_-(\radlim)=\mu_+(\radlim)=1$.
\item[(ii)] $(\quotient{\Gamma}{ \SX}, g_\Gamma, 
m_\Gamma)$ is  conservative.
\item[(iii)] 
$(\quotient{\Gamma}{ \SX}, g_\Gamma, 
m_\Gamma)$ is ergodic and $m_\Gamma$ is not supported on a single divergent orbit. \end{enumerate}
Moreover, each of the three statements implies that $m_\Gamma$ is equal to the weak Ricks' measure $\overline m_\Gamma$ on $\quotient{\Gamma}{[\SX]}$  and  to any Knieper's measure on $\quotient{\Gamma}{\SX}$ associated to $\overline\mu$ (if it exists). 
\end{theorem}

We finally mention a result concerning the dynamical systems $(\ndpt\SX, \Gamma, \overline\mu)$ and $(\rand\times\rand,\Gamma, \mu_-\otimes \mu_+)$ first introduced in Section~\ref{geodcurrentmeasures}. From the construction of the Ricks'  measure $m_\Gamma$ associated to the quasi-product geodesic current $\overline\mu$ defined on $\ndpt\reg$ which is absolutely continuous with respect to the product  $ ( \mu_-\otimes \mu_+)\ein_{\ndpt\reg}$  of non-atomic probability measures $\mu_{\pm}$ on $\rand$ with $\supp(\mu_{\pm})=\Lim$ we immediately get
\begin{lemma} 
$(\quotient{\Gamma}{ \SX}, g_\Gamma, 
m_\Gamma)$ is ergodic if and only if $(\ndpt\SX, \Gamma, \overline\mu)$ is ergodic if and only if $(\rand\times\rand,\Gamma, \mu_-\otimes \mu_+)$ is ergodic.
\end{lemma}

\section{Geodesic currents coming from a conformal density}\label{currentsfromconfdens}

For the remainder of this article we will specialize to a particular kind of geodesic currents, namely the ones arising from a conformal density.
As before  $\XX$ will denote a proper Hadamard space  and $\Gamma<\is(\XX)$ a discrete rank one group. We further fix a base point $\xo\in\XX$ on an invariant geodesic of a rank one element in $\Gamma$.

We start with an important definition: Since $\Gamma<\is(\XX)$ is discrete and $\XX$ is proper the {\hl orbit counting function}
\[ N_\Gamma(R):=\#\{\gamma\in\Gamma\colon d(\xo,\gamma \xo)\leq R\}\]
is finite for all $R>0$.  The number 
\[ \delta_\Gamma =\limsup_{R\to +\infty}\frac{\ln\bigl(N_\Gamma(R)\bigr)}{R}\]
is called the {\hd critical exponent} of $\Gamma$; it is independent of the choice of base point $\xo\in\XX$ and satisfies the equality 
\begin{equation}\label{critexp} \delta_\Gamma= \inf\{s>0\colon \sum_{\gamma\in\Gamma} \e^{-s d(\xo,\gamma \xo)} \ \text{ converges}\}.\end{equation}
A discrete group $\Gamma$ is said to be {\hd divergent} if
\[\sum_{\gamma\in\Gamma} \e^{-\delta_\Gamma d(\xo,\gamma\xo)}\quad\text{diverges},\]
and {\hd convergent} otherwise (that is when the infimum in (\ref{critexp}) is attained).

Given $\delta\ge 0$, a $\delta$-dimensional  $\Gamma$-invariant  conformal density is a continuous map $\mu\,$ of $\XX$ into the cone of positive finite Borel
measures on $\rand$ \st $\mu_\xo:=\mu(\xo)$ is supported on the limit set $\Lim$, $\mu\,$ is $\Gamma$-equivariant (that is $\gamma_*\mu_x=\mu_{\gamma x}$ for
all $\gamma\in\Gamma$, $x\in\XX$)\footnote{Here $\gamma_*\mu_x$ denotes the measure defined by $\gamma_*\mu_x(E)=\mu_x(\gamma^{-1}E)$ for any Borel set $E\subseteq\rand$.}  and
\begin{equation}\label{conformality}
  \frac{\d \mu_x}{\d \mu_\xo}(\eta)=\e^{\delta
\bs_{\eta}(\xo,x)} \quad\mbox{for any}\ \,x\in\XX\ \text{and }\  \eta\in\supp(\mu_\xo).
\end{equation}

The existence of  a  $\delta$-dimensional   $\Gamma$-invariant conformal density for $\delta=\delta_\Gamma$ goes back to  S.~J.~Patterson \mbox{(\cite{MR0450547})} in the case of Fuchsian groups, and 
it turns out that his explicit construction extends to arbitrary discrete isometry groups of Hadamard spaces with positive critical exponent
(see for example \cite[Lemma 2.2]{MR1465601}). 
This condition is satisfied for any discrete rank one group  $\Gamma< \is(\XX)$ as it contains by definition a non-abelian free subgroup generated by two independent rank one elements. 

We now fix $\delta>0$ and let $\mu=(\mu_x)_{x\in\XX}$  be a
  $\delta$-dimensional $\Gamma$-invariant conformal density.   By definition of a conformal density we have
$0<\mu_\xo(\rand)<\infty$, and we will assume
that $\mu_\xo$ is normalized \st $\mu_\xo(\rand)=1$.  
  
 Before we construct a geodesic current from a conformal density we want to list a few results concerning these.

We first turn our attention to the radial limit set defined by~(\ref{radlimpoint}). 
Recall that for $y\in \XX$ and $r>0$ $B_y(r)\subseteq\XX$  denotes the open ball of radius $r$ centered at $y\in\XX$.
 If $x\in \XX$ we define the {\hd shadow} 
\[ {\mathcal O}_{r}(x,y):=\{\eta \in\rand\colon \sigma_{x,\eta}(\RR_+)\cap B_y(r)\neq\emptyset\};\]
if $\xi\in\rand$ we set
\[ {\mathcal O}_{r}(\xi,y):=\{\eta \in\rand\colon \exists \ v\in\ndpt^{-1}(\xi,\eta) \quad \text{with }\ v(0)\in B_y(r)\}.\]
Notice that with these definitions the radial limit set can be written as
\[ \radlim=\bigcup_{c>0}\bigcap_{R>1} \bigcup_{\begin{smallmatrix}{\scriptscriptstyle\gamma\in\Gamma}\\{\scriptscriptstyle d(\xo,\gamma\xo)>R}\end{smallmatrix}}{\mathcal O}_c(\xo,\gamma\xo);\]
again, the definition is independent of the choice of base point $\xo\in\XX$.

One corner stone result concerning $\delta$-dimensional  $\Gamma$-invariant conformal densities is Sullivan's shadow lemma which gives an asymptotic estimate for the measure of  the shadows
${\mathcal O}_r(\xo,\gamma\xo)$ as $d(\xo,\gamma\xo)$ tends to infinity; obviously this will lead to  estimates for the measure of the radial limit set. 
We will need here an extension of  the shadow lemma \cite[Lemma~3.5]{MR2290453}  to the  following refined versions of the shadows above which were first introduced  by  T.~Roblin (\cite{MR2057305}):
For
$r>0$, $c>0$  and $x,y\in\XX$ 
 we set
\begin{align*}
{\mathcal O}^+_{r,c}(x,y) &:= \{\xi\in\rand \colon \exists\, z\in B_x(r)\ \st
\sigma_{z,\xi}(\RR_+)\cap B_y(c)\neq\emptyset\},\nonumber \\
{\mathcal O}^-_{r,c}(x,y) &:= \{\xi\in\rand\colon \forall\, z\in B_x(r)\
\mbox{we have}\ \sigma_{z,\xi}(\RR_+)\cap B_y(c)\neq\emptyset\}.\nonumber 
\end{align*}
It is clear from the definitions that 
\begin{equation}\label{shadrelation}
{\mathcal O}^-_{r,c}(x,y)=\bigcap_{z\in B_x(r)} {\mathcal O}_{c}(z,y) \subset{\mathcal O}_{c}(x,y)\subseteq \bigcup_{z\in B_x(r)} {\mathcal O}_{c}(z,y)={\mathcal O}^+_{r,c}(x,y);\end{equation}
moreover, ${\mathcal O}^-_{r,c}(x,y)$ is non-increasing in $r$ and non-decreasing in $c$. We further have the following generalization of Sullivan's shadow lemma:
\begin{proposition}\cite[Proposition 3 and Remark 3]{LinkPicaud}\label{shadowlemma}\
Let $\XX$ be a proper Hadamard space and $\Gamma<\is(\XX)$ a discrete rank one group. Let  $\delta>0 $  and $\mu\,$ a  $\delta$-dimensional $\Gamma$-invariant
conformal density. Then for any $r>0$ there exists a
constant $c_0\ge r$ with the following property: If $c\geq c_0$ there
exists a constant $D=D(c)>1$ \st for all $\gamma\in\Gamma$ with
$d(\xo,\gamo)>2c$ we have
$$ \frac1{D}\; \e^{-\delta d(\xo,\gamma \xo)}\le \mu_\xo\big({\mathcal O}_{r,c}^-(\xo,\gamma\xo)\big)\le \mu_\xo\big({\mathcal O}_c(\xo,\gamo)\big)\le \mu_\xo\big({\mathcal O}_{c,c}^+(\xo,\gamo)\big)\le D \e^{-\delta d(\xo,\gamma\xo)}.$$ 
Moreover, the upper bound holds for {\hl all} $\gamma\in\Gamma$.
\end{proposition}
The proof of this proposition  in the special case of a Hadamard {\hl manifold} $\XX$ was given in \cite{LinkPicaud}; however the proof  there does not use the fact that $\XX$ is a manifold. 

Next we state some results from Section~3 in  \cite{MR2290453} 
and from Section~5 in \cite{LinkPicaud} which all rely on the shadow lemma above and which remain valid in the setting of non-Riemannian Hadamard spaces. 

\begin{lemma}\label{confdensexistence}\cite[Proposition~3.7]{MR2290453}
If $\mu$ is a $\delta$-dimensional $\Gamma$-invariant conformal density, then $\delta\ge\delta_\Gamma$.
\end{lemma}
\begin{lemma}\cite[Lemma 5.1]{LinkPicaud}\label{convseries}\
If $\ \displaystyle \sum_{\gamma\in\Gamma} \e^{-\delta  d(\xo,\gamo)}$ converges, then  $\mu_\xo(\radlim)=0$.
\end{lemma}
\noindent 
In particular, if $\delta>\delta_\Gamma$, then from  (\ref{critexp}) we immediately get $\mu_\xo(\radlim)=0$. 

Notice that the converse statement to Lemma~\ref{convseries} is much more intricate; we will have to postpone its proof to Section~\ref{divergentmeansconservative} as we will need to work with a weak Ricks' measure on $\quotient{\Gamma}{[\SX]}$. 

The following lemma states that $\Gamma$ acts ergodically on the radial limit set with respect to the measure class defined by $\mu$:
\begin{lemma}\cite[Proposition 4]{LinkPicaud}\label{ergodicity}\
If $A\subseteq \radlim$ is a $\Gamma$-invariant Borel subset of $\radlim$, then
$\mu_\xo(A)=0$ or $\mu_\xo(A)=\mu_\xo(\rand)=1$.
\end{lemma}
By a standard argument (see for example the proof of Theorem~4.2.1 in \cite{MR1041575}) we get the following
\begin{corollary}\label{uniqueness}
If $\mu_\xo(\radlim)>0$ then $\delta=\delta_\Gamma$ and $\mu\, $ is the unique $\delta_\Gamma$-dimensional $\Gamma$-invariant conformal density normalized such that $\mu_\xo(\rand)=1$.
\end{corollary}
Finally, the following statement clarifies the possible existence of atoms:

\begin{proposition}\cite[Proposition 5]{LinkPicaud}\label{atomicpart}\
A radial limit point cannot be a point mass for a  $\delta$-dimensional  $\Gamma$-invariant conformal density $\mu$.
 \end{proposition}
 
We are now going to  construct a geodesic current from a $\delta$-dimensional\break  $\Gamma$-invariant conformal density. Notice that according to Lemma~\ref{confdensexistence} such a density only exists if $\delta\ge\delta_\Gamma$. 

First we define for $y\in \XX$ a map 
\[\Gr_y:\rand\times\rand\to\RR,\quad (\xi,\eta)\mapsto \frac12 \sup_{x\in\XX} \bigl(\bs_\xi(y,x)+\bs_\eta(y,x)\bigr).\]
Obviously, the map $\Gr_y$ has values in $[0,\infty]$, and comparing it to the definition by R.~Ricks following \cite[Lemma~5.1]{Ricks} we have the relation
$ \Gr_y(\xi,\eta)=-2 \beta_y(\xi,\eta)$ for all $(\xi,\eta)\in\rand\times\rand$. Hence according to Lemma~5.2 in \cite{Ricks} $\Gr_y(\xi,\eta)$ is finite if and only if $(\xi,\eta)\in\ndpt\SX$; moreover, 
\begin{equation}\label{GromovProd}
\Gr_y(\xi,\eta)=\frac12\bigl(\bs_\xi(y,z)+\bs_\eta(y,z)\bigr) 
\end{equation}
if and only if $z\in (\xi\eta)$ lies on the image of a geodesic joining $\xi$ and $\eta$. 
So the map $\Gr_y$ extends 
 the {\hd Gromov product} defined 
in \cite{MR1341941} via the formula~(\ref{GromovProd}) from $\ndpt\SX$ to $\rand\times\rand$. By Lemma~5.3 in \cite{Ricks} $\Gr_y$ is continuous on $\ndpt\reg$ and lower semicontinuous on $\rand\times\rand$.

We now define as in Section~7 of  \cite{Ricks} a measure  $\overline{\mu}\,$ on  
$\ndpt\SX \subseteq\rand\times\rand$ 
 via
\begin{equation}\label{overlinemudef} d\overline{\mu}(\xi,\eta)=\e^{2\delta \Gr_\xo(\xi,\eta)} {\mathbbm 1}_{\ndpt\mathcal{R}}(\xi,\eta)\d\mu_\xo(\xi)\d\mu_\xo(\eta).\end{equation}
As $\ndpt\SX$ is locally compact and as $\overline\mu\,$ is finite for all compact subsets of $\ndpt\SX$, the measure $\overline\mu\,$ is Radon; it is non-trivial by~(\ref{overlinemunotzero}). 
Moreover, 
 $\Gamma$-equivariance and conformality~(\ref{conformality}) of the $\delta$-dimensional  $\Gamma$-invariant conformal density $\mu=(\mu_x)_{x\in\XX}$ occurring in the formula imply that $\overline\mu\,$ is invariant by the diagonal action of $\Gamma$ (and also independent of the choice of $\xo\in\XX$).  

Hence as described at the end of Section~\ref{geodcurrentmeasures}  we can construct from the geodesic current $\overline\mu\,$  Knieper's measure $m_\Gamma$ (provided $\overline\mu\,$ is supported on $\ndpt\zero$ or, more generally, if there exists a geodesic flow invariant Borel  measure $\lambda_{(\xi\eta)}$ on the set $(\xi\eta)\subseteq\XX$ for $\overline\mu$-almost every $(\xi,\eta)\in\ndpt\SX$) and both Ricks' weak measure $\overline m_\Gamma$ on $\quotient{\Gamma}{[\SX]}$ and Ricks' measure $m_\Gamma^0$ on $\quotient{\Gamma}{\SX}$ (which will be trivial if $\overline\mu(\ndpt\zero)=0$).

Combining Lemma~\ref{convseries} with Lemma~\ref{consdiss} (b) we get the following
\begin{proposition}\label{confgivesdissipative} If $\delta>\delta_\Gamma$ or if $\,\Gamma$ is convergent, then $\overline\mu( \ndpt\SX_{\Gamma}^{\small{\mathrm{rad}}})=0$, and hence\break the dynamical systems $\bigl(\quotient{\Gamma}{\SX}, (g^t_\Gamma)_{t\in\RR},  m_\Gamma\bigr)$ with Knieper's measure $m_\Gamma$ and\break $\bigl(\quotient{\Gamma}{ [\SX]}, (g^t_\Gamma)_{t\in\RR}, \overline m_\Gamma\bigr)$ with the weak Ricks'  measure $\overline m_\Gamma$ associated to $\overline \mu\,$ are dissipative and non-ergodic unless $\overline \mu$  is supported on a single orbit $\, \Gamma\cdot (\xi,\eta)\subseteq\ndpt\SX$. 
\end{proposition}

Notice that if $\XX$ is a proper CAT$(-1)$-space and    $\Gamma<\is(\XX)$ a non-elementary discrete group, then  the so-called {\hl Bowen-Margulis measure}  (see for example \cite[p.12]{MR2057305} or  \cite[Section~3]{MR1207579}) on $\quotient{\Gamma}{\SX}$ -- which in this case equals $\quotient{\Gamma}{\zero}$ --  is precisely Knieper's measure $m_\Gamma$ or equivalently Ricks'  measure $m_\Gamma^0$  associated to the geodesic current $\overline\mu$.

We finally mention a few further properties of the quasi-product geodesic current $\overline\mu\,$ defined by (\ref{overlinemudef}).
First, as $v(0)\in B_\xo(R)$ implies $\Gr_\xo(v^-,v^+)\le R$, 
we have
\[ \overline\mu\bigl(\ndpt\{v\in\SX\colon v(0)\in B_\xo(R)\}\bigr)\le \e^{2\delta  R} \]
for all $R>0$; hence 
\[  \Delta\stackrel{(\ref{boundongrowth})}{ =}\sup \Big\{ \frac{\ln  \overline\mu\bigl(\ndpt {\mathcal B}(R)\bigr)}{R}\colon R>0\Big\}=2\delta.\] 

Second, if $\mu_\xo(\radlim)=\mu_\xo(\rand)=1$, then $\mu_\xo$ is non-atomic by Proposition~\ref{atomicpart}. So according to Lemma~\ref{regnotnecessary}  the geodesic current $\overline\mu\,$ is given by
\begin{equation}\label{overlinemudefconfcons} d\overline{\mu}(\xi,\eta)=\e^{2\delta \Gr_\xo(\xi,\eta)} (\xi,\eta)\d\mu_\xo(\xi)\d\mu_\xo(\eta),\end{equation}
that is the factor ${\mathbbm 1}_{\ndpt\mathcal{R}}$ in (\ref{overlinemudef}) can be removed. Moreover, all the equivalent statements of  Theorem~\ref{conservativestatement} hold.

 \section{Conservativity in the case of divergent groups}\label{divergentmeansconservative}

 As before, $\XX$ will be a proper Hadamard space, $\Gamma$ a discrete rank one group and $\xo\in\XX$ a fixed base point on an invariant geodesic of a rank one element in $\Gamma$.  

 The goal of this section is to prove the converse statement to Lemma~\ref{convseries}, that is if \\[-2mm]
 \[\ \displaystyle \sum_{\gamma\in\Gamma} \e^{-\delta  d(\xo,\gamo)} \quad\text{diverges}, \quad \text{then }\quad \mu_\xo(\radlim)>0.\] 
 However, by  Lemma~\ref{confdensexistence} a $\delta$-dimensional $\Gamma$-invariant conformal density $\mu$ only exists if $\delta\ge \delta_\Gamma$; for $\delta>\delta_\Gamma$ the
Poincar{\'e} series 
 \[ \sum_{\gamma\in\Gamma} \e^{-\delta d(\xo,\gamma\xo)}\]
converges according to the alternative definition~(\ref{critexp}) of the critical exponent of $\Gamma$.  So from here on we will assume that $\Gamma$ is divergent and that $\mu=(\mu_x)_{x\in\XX}$ is a\break $\delta_\Gamma$-dimensional $\Gamma$-invariant conformal density. 

 In order to prove that the radial limit set of $\Gamma$ has full measure with respect to $\mu_\xo$ 
 we follow as in \cite[Section~6]{LinkPicaud} Roblin's exposition. As we want to apply the generalization of the second Borel-Cantelli lemma Lemma~2 in \cite{MR0766098}, we need to work with a weak Ricks' measure $\overline m_\Gamma$ on $\quotient{\Gamma}{[\SX]}$ and find an appropriate Borel set $\overline K\subseteq [\SX]$ whose projection to $\quotient{\Gamma}{[\SX]}$ has finite $\overline m_\Gamma$-measure and which satisfies the two Renyi inequalities~(\ref{upbound}) and~(\ref{lowbound}) below.  Notice that in order to get a better control -- and a proof even without the presence of a {\hl zero width} rank one element -- apart from using the weak Ricks' measure we need to choose the set $\overline K$ more carefully than in \cite[Section~6]{LinkPicaud}.

Before we proceed we need a result concerning the following slightly refined version of the corridors first introduced  by  T.~Roblin (\cite{MR2057305}):
For
$r>0$, $c>0$  and $x,y\in\XX$ 
 we set
\begin{align}
{\mathcal L}_{r,c}(x,y) &:=  \{(\xi,\eta)\in\ndpt\SX\colon \exists\, v\in\ndpt^{-1}(\xi,\eta)\ \exists\, t>0\ \st\nonumber \\
&\hspace*{4.8cm}v(0)\in B_x(r),\ v( t)\in B_y(c) \}.\label{Lrc}
\end{align}
Notice that in the case of a Hadamard manifold the definition is equivalent to the one given in Section~2 of \cite{LinkPicaud}; however, due to the fact that the extension of a geodesic segment to a geodesic line is in general not unique in a singular Hadamard space the definition~(8) given there is not convenient here.

It is clear from the definitions that   ${\mathcal L}_{r,c}(x,y) $
is non-decreasing in both $r$ and $c$. Moreover,  
for all $r',c'>0$, $x'\in B_x(r')$ and $y'\in B_y(c')$ with $d(x',y')>r+r'+c+c'$ we have
\begin{equation}\label{lrcinclus}
{\mathcal L}_{r,c}(x,y)\subseteq  {\mathcal L}_{r+r',c+c'}(x',y'),
\end{equation}
and the following result from \cite{LinkPicaud} (whose proof extends to non-Riemannian Hada\-mard spaces)  asserts that for suitable $r$ and $c$ the sets
$ {\mathcal L}_{r,c}(\xo,\gamma \xo)$ are big enough for all but a finite number of elements in $\Gamma$.  Recall that $\Gamma<\is(\XX)$ was assumed to be a discrete rank one group and that the base point   $\xo$ belongs to an invariant geodesic of a rank one element $h\in\Gamma$. 
\begin{proposition}\cite[Proposition~1]{LinkPicaud}\label{lrcinproduct}\
Let  $r_0>\width(h)$ and $U^-,U^+\subseteq\ganz$ the open disjoint neighborhoods of $h^-$, $h^+$ provided by Lemma~\ref{joinrankone} for $r_0$. Then there exists a finite set
$\Lambda\subseteq\Gamma$ such that the following holds:

For any $c>0$ there exists $R\gg 1$ \st if
$\gamma \in \Gamma$ satisfies $d(\xo,\gamma \xo)>R$, then for some  $\beta\in\Lambda$ we have
$$ {\mathcal L}_{r,c}(\xo,\beta\gamma \xo)\cap \big(U^-\times U^+\big)\supseteq (U^- \cap \rand) \times {\mathcal O}_{r,c}^-(\xo,\beta\gamma\xo)\qquad\mbox{for all }\ r\ge r_0.$$
\end{proposition}

 We 
fix $r=r_0>\width(h)$  and open disjoint neighborhoods $U^-,U^+\subseteq\ganz$  of $h^-,h^+$ provided by Lemma~\ref{joinrankone} for $r_0$. Let $\Lambda\subseteq\Gamma$ be the finite subset provided by Proposition~\ref{lrcinproduct}. We then set
\[ \rho:=\max\{d(\xo,\beta\xo)\colon \beta\in\Lambda\}\]
and -- with the constant $c_0>r$ from the shadow lemma Proposition~\ref{shadowlemma} -- fix
\[ c>c_0+\rho.\]
Notice that by choice of $c_0>r=r_0>\width(h)$ we always have $c>\width(h)$.

For this fixed constant $c$ and with the sets $U^-,U^+\subseteq\ganz$ as above we 
 define 
\begin{equation}\label{Kdef}
\overline K:=\{ g^s [v] \colon v \in \SX,\ v(0)\in B_\xo(c),\   (v^-,v^+)\in\Gamma (U^-\times U^+),\ s\in \bigl(-\frac{c}2,\frac{c}2\bigr) \},
\end{equation}
which is an open subset of $[\reg]$. Moreover, every representative $u\in\SX$ of $[u]\in\overline K$ satisfies $\width(u)\le 2c$:  Indeed,  $[u]\in\overline K$ implies that   $\alpha  u^-\in U^-$ and $ \alpha u^+\in U^+$ for some $\alpha\in\Gamma$; hence by Lemma~\ref{joinrankone} the geodesic $\alpha\cdot u\in\SX$ is rank one and $\width(\alpha\cdot u)\le 2c$. The claim then follows from $\is(\XX)$-invariance of the width function.

We further remark that by construction every orbit of the geodesic flow which enters $\overline K$ spends at least time $c$ and at most time $3c$ in it. 

In order to make the exposition of the proof of Proposition~\ref{divseries} below more transparent, we first state a few easy geometric estimates concerning intersections of the form 
 \[\overline K\cap g^{-t}\gamma \overline K\qquad\text{and}\quad \overline K\cap g^{-t}\gamma \overline K\cap g^{-s-t}\varphi \overline K\]
in $[\SX]$ 
with $t,s>0$ and $\gamma,\varphi\in \Gamma$. 
The first one gives a relation to  the sets ${\mathcal L}_{c,c}(\xo,\gamma\xo)$ introduced in~(\ref{Lrc}):
\begin{lemma}\label{K0calL}
\begin{align*} {\mathcal L}_{c,c}(\xo,\gamma\xo)\cap \Gamma(U^-\times U^+)&\subseteq \ndpt\big(\{ \overline K\cap g^{-t}\gamma \overline K\colon t>0\}\big)\\
&\subseteq {\mathcal L}_{2c,2c}(\xo,\gamma\xo)\cap \Gamma(U^-\times U^+)\end{align*}
\end{lemma}
\begin{proof} For the first  inclusion we let $(\xi,\eta)\in {\mathcal L}_{c,c}(\xo,\gamma\xo) \cap \Gamma(U^-\times U^+)$ be arbitrary. Then there exists $\alpha\in\Gamma$ \st $(\xi,\eta)\in \alpha(U^-\times U^+)$,  and by definition~(\ref{Lrc}) there exists $v\in \SX$ with $(v^-,v^+)=(\xi,\eta)$, $d(\xo, v(0))<c$ and $d(\gamma\xo, v(t))<c$  for some $t>0$. We conclude that $[v]\in\overline K$ and, since $\gamma^{-1}(v^-,v^+)\in \gamma^{-1}\alpha(U^-\times U^+)\subseteq \Gamma (U^-\times U^+)$,  also $\gamma^{-1}g^t [v]\in \overline K$. 

For the second inclusion we let $(\xi,\eta)\in \ndpt \bigl(\{ \overline K\cap g^{-t}\gamma \overline K\colon t>0\}\bigr)$. 
Then $(\xi,\eta)\in\Gamma(U^-\times U^+)$ and there exist $v,u \in \ndpt^{-1}(\xi,\eta)$, $v\sim u$ \st   $v(0)\in B_\xo(c)$ 
and  $(g^t u)(0)\in B_{\gamma\xo}(c)$  for some $t>0$. Since $\xi\in \alpha U^-$ and $\eta\in \alpha U^+$ for some $\alpha \in\Gamma$ we know from Lemma~\ref{joinrankone} (since $c>\width(h)$ and $\xo$ was chosen on an invariant geodesic of the rank one element $h$) that every rank one geodesic $w\in\mathcal{R}$ joining $\alpha^{-1}\xi$ and $\alpha^{-1}\eta$ 
has $\width(w)\le 2c$. Now both $\alpha^{-1}v$ and $\alpha^{-1}u$ are such rank one geodesics and therefore we get from $u\sim  v$ 
\[ d\bigl(u(s), v(s)\bigr)=  d\bigl(\alpha^{-1}u(s), \alpha^{-1} v(s)\bigr)\le 2c\quad\text{for all }  \ s\in\RR.\]
Choosing $w\in \SX$ with $w\sim v$ \st 
\[d\bigl(u(s), w(s)\bigr)=d\bigl(w(s), v(s)\bigr)=\frac12 d\bigl(u(s), v(s)\bigr) \le c\]  for all $s\in\RR$  we conclude that 
$\ (\xi,\eta)=(w^-,w^+)\in {\mathcal L}_{2c,2c}(\xo,\gamma\xo).$
\end{proof}

As a direct consequence we obtain that for all $t,s>0$ and all $\gamma,\varphi\in \Gamma$
\begin{eqnarray}\label{INC2} \ndpt\big(\overline K\cap g^{-t}\gamma \overline K\cap g^{-t-s}\varphi \overline K\big)\subseteq
{\mathcal L}_{2c,2c}(\xo,\varphi\xo)\cap \Gamma(U^-\times U^+).\end{eqnarray}

The following geometric estimate gives a relation between the constants $t,s>0$ and the elements $\gamma,\varphi\in\Gamma$:
\begin{lemma}\label{IT} 
$\overline K\cap g^{-t}\gamma \overline K\ne\emptyset\ $ implies  
\[ |d(\xo,\gamma\xo)-t|\le 5c, \]
and $\overline K\cap g^{-t}\gamma \overline K\cap g^{-s-t}\varphi \overline K\ne \emptyset\ $ further gives 
\[  0\leq d(\xo,\gamma\xo)+d(\gamma\xo,\varphi\xo)-d(\xo,\varphi\xo)\le 15 c.\]
\end{lemma}
\begin{proof}
Assume that $\overline K\cap g^{-t}\gamma \overline K\ne \emptyset$. Then there exist $v,u\in \SX$ with $v\sim u$,  $(v^-,v^+)=(u^-,u^+)\in\Gamma(U^-\times U^+)$ and 
$s,r\in  (-c/2,c/2)$ \st \[  (g^s v)(0)=v(s)\in B_\xo(c)  \quad\text{and }\ (g^{r}g^t u)(0)=u(r+t)\in B_{\gamma\xo}(c).\]  
So in particular -- as in the proof of the second inclusion above -- we get 
\[ d\bigl(u(s), v(s)\bigr)\le 2c\quad\text{for all }\  s\in\RR.\]
Hence 
\begin{align*} 
d(\xo,\gamma\xo)&\le d(\xo, v(s))+d(v(s),v(0))+ d(v(0),v(t))+d(v(t), u(t))\\
&\qquad +d(u(t),u(r+t))+d(u(r+t),\gamma\xo)\le c+s+t+2c+r+ c\le t+ 5 c 
\end{align*}
and similarly the reverse inequality
\[ d(\xo,\gamma\xo)\ge t- 5c.\]

If $ \overline K\cap g^{-t}\gamma \overline K\cap g^{-s-t}\varphi \overline K\ne \emptyset$, then from the first claim we get 
\[ |d(\xo,\gamma\xo)-t|\le 5c,\quad |d(\xo,\varphi\xo)-s-t|\le 5c \quad\text{and }\quad |d(\gamma\xo,\varphi\xo)-s|\le 5c.\]
So we conclude again by the triangle inequality.
\end{proof}
Finally we remark that  if  $(\xi,\eta)\in{\mathcal L}_{2c,2c}(\xo,\varphi\xo)$,
then there exists $z\in (\xi\eta)\cap B_{\xo}(2c)$ \st
\[ \Gr_{\xo}(\xi,\eta)=\frac12\big(\bs_\xi(\xo,z)+\bs_\eta(\xo,z)\big)\]
which immediately gives the estimate
\begin{equation}\label{GROMOV}  \Gr_{\xo}(\xi,\eta) \le 2c.\end{equation}

Recall that $\mu$ is a $\delta_\Gamma$-dimensional $\Gamma$-invariant conformal density. Let $\overline\mu$ be the geodesic current on $\ndpt\SX$ given by the formula~(\ref{overlinemudef}) and $\overline m_\Gamma$ the induced weak Ricks' measure on $\quotient{\Gamma}{[\SX]}$ (which is supported on $\quotient{\Gamma}{[\reg]}$). Notice that for the projection $\overline K_\Gamma\subseteq \quotient{\Gamma}{[\reg]}$ of the set $\overline K\subseteq [\reg]$ defined in (\ref{Kdef}) to $\quotient{\Gamma}{[\reg]}$ we have
\[ 0< \overline m_\Gamma(\overline K_\Gamma)\le \overline m(\overline K)\le 3c\cdot \e^{2c\delta } \underbrace{ (\mu_\xo\otimes\mu_\xo)\bigl(\Gamma (U^-\times U^+)\bigr)}_{\le 1}<\infty.\]

We are now going to prove the converse to Lemma~\ref{convseries} in our setting of a proper Hadamard space $\XX$ and a discrete rank one group $\Gamma<\is(\XX)$. Our result here generalizes Proposition~1 in \cite{LinkPicaud} as we neither require $\XX$ to be a {\hl manifold} nor $\Gamma$ to contain a {\hl strong} rank one isometry or a zero width rank one isometry. 

\begin{proposition}\label{divseries}\
If  $\ \sum_{\gamma\in\Gamma} \e^{-\delta_\Gamma d(\xo,\gamo)}\,$ diverges, then 
$\mu_\xo(\radlim)>0$. 
\end{proposition}
\begin{proof}  We argue by contradiction, assuming that the sum $\ \sum_{\gamma\in\Gamma} \e^{-\delta_\Gamma d(\xo,\gamo)}$ diverges and that $\,\mu_\xo(\radlim)=0$. We will show that for the  Borel set  $\overline K\subseteq [\reg]$  defined by (\ref{Kdef})
 the following inequalities hold for $T$ sufficiently  large with universal constants $C,C'>0$:

\begin{equation}\label{upbound}
\int_{0}^T \d t \int_0^T \d s \sum_{\gamma,\varphi\in\Gamma} \overline m(\overline K\cap g^{-t}\gamma \overline K\cap g^{-t-s}\varphi \overline K) \le  C
\biggl(\sum_{\begin{smallmatrix}{\scriptscriptstyle \gamma\in\Gamma}\\
        {\scriptscriptstyle d(o,\gamo)\le T}\end{smallmatrix}} \e^{-\delta_\Gamma
    d(\xo,\gamma\xo)}\biggr)^2
\end{equation}
\begin{equation}\label{lowbound}
\int_{0}^T \d t  \sum_{\gamma\in\Gamma} \overline m(\overline K\cap g^{-t}\gamma \overline K) \ge  C' \sum_{\begin{smallmatrix}{\scriptscriptstyle\gamma\in\Gamma}\\
        {\scriptscriptstyle d(o,\gamo)\le T}\end{smallmatrix}} \e^{-\delta_\Gamma d(\xo,\gamma\xo)}
\end{equation}

Once these inequalities are proved and under the assumption that the sum $\ \sum_{\gamma\in\Gamma}\e^{-\delta_\Gamma d(\xo,\gamma\xo)}$ diverges one can apply the above mentioned
generalization of the second Borel-Cantelli lemma, and the conclusion follows as in
\cite[p.~20]{MR2057305} (applying \cite[Lemma~2]{MR0766098} to the finite measure $M=\overline m_\Gamma$ restricted to $\overline K_\Gamma\subseteq \quotient{\Gamma}{[\reg]} $), namely 
$$\overline m_\Gamma\bigl(\{ [v]\in \quotient{\Gamma}{[\SX]}  \colon \int_0^\infty \chr_{\overline K_\Gamma\cap g^{-t}_\Gamma \overline K_\Gamma}([v])=\infty\}\bigr)>0.$$
This means that the dynamical system
$\bigl(\quotient{\Gamma}{[\SX]}, g_\Gamma, 
\overline m_\Gamma\bigr)$ is not  dissipative. But by Lemma~\ref{consdiss} (b) this is a contradiction to $\mu_\xo(\radlim)=0$.

We begin with the proof of (\ref{upbound}):
From the definition of the weak Ricks' measure 
and the estimates~(\ref{INC2}) and (\ref{GROMOV})  it follows that for   all
$\gamma,\varphi\in \Gamma$
\begin{align*}
\overline m(\overline K\cap g^{-t}\gamma \overline K\cap g^{-t-s}\varphi \overline K) &\le \int_{{\mathcal L}_{2c,2c}(\xo,\varphi\xo)\cap \Gamma(U^-\times U^+)} \d \mu_\xo(\xi)  \d \mu_\xo(\eta) \e^{2 \delta_\Gamma
  \Gr_\xo(\xi,\eta)} \cdot c  \\
& \le \e^{4 c \delta_\Gamma} c  \int_{{\mathcal L}_{2c,2c}(\xo,\varphi\xo)\cap \Gamma(U^-\times U^+)} \d \mu_\xo(\xi)  \d \mu_\xo(\eta) .
\end{align*}
Since obviously $ {\mathcal L}_{2c,2c}(\xo,\varphi\xo)\cap \Gamma(U^-\times U^+)\subseteq {\mathcal L}_{2c,2c}(\xo,\varphi\xo)\subseteq \rand\times {\mathcal O}^+_{2c,2c}(\xo,\varphi\xo)$ we obtain
\begin{align*}
\overline m(\overline K\cap g^{-t}\gamma \overline K\cap g^{-t-s}\varphi \overline K) & \le \e^{4 c \delta_\Gamma} c   \mu_\xo\bigl({\mathcal O}^+_{2c,2c}(\xo,\varphi\xo)\bigr) \le  \e^{4 c \delta_\Gamma} c  D(c) \e^{-\delta_\Gamma d(\xo,\varphi\xo) }, 
\end{align*}
where we used the shadow lemma Proposition~\ref{shadowlemma} in the last step.

Using Lemma~\ref{IT} we finally get
\begin{align*} \int_0^T \d t\int_0^T \d s  & \hspace{1mm} \overline m(\overline K\cap  g^{-t}\gamma \overline K\cap g^{-t-s}\varphi \overline K)\le (10c)^2 \sum_{\begin{smallmatrix}{\scriptscriptstyle\gamma,\varphi\in\Gamma}\\{\scriptscriptstyle d(o,\gamo)\le T+5c}\\{\scriptscriptstyle
          d(\gamo,\varphi\xo)\le T+5c}\end{smallmatrix}} \e^{4 c \delta_\Gamma} c  D(c) \e^{-\delta_\Gamma d(\xo,\varphi\xo) }\\
& \le 100c^3  \e^{4 c \delta_\Gamma}   D(c)  \sum_{\begin{smallmatrix}{\scriptscriptstyle\gamma,\varphi\in\Gamma}\\{\scriptscriptstyle d(o,\gamo)\le T+5c}\\{\scriptscriptstyle
          d(\gamo,\varphi\xo)\le T+5c}\end{smallmatrix}} \e^{-\delta_\Gamma (d(\xo,\gamma\xo) + d(\gamma\xo,\varphi\xo)-15c) } \\
& = 100c^3  \e^{19 c \delta_\Gamma}  D(c)  \sum_{\begin{smallmatrix}{\scriptscriptstyle\gamma,\alpha\in\Gamma}\\{\scriptscriptstyle d(o,\gamo)\le T+5c}\\{\scriptscriptstyle
          d(\xo,\alpha\xo)\le T+5c}\end{smallmatrix}} \e^{-\delta_\Gamma (d(\xo,\gamma\xo) + d(\xo,\alpha\xo)) } \\
& = 100 c^3  \e^{19 c \delta_\Gamma}  D(c) \Bigl( \sum_{\begin{smallmatrix}{\scriptscriptstyle\gamma\in\Gamma}\\{\scriptscriptstyle d(o,\gamo)\le T+5c}\end{smallmatrix}} \e^{-\delta_\Gamma d(\xo,\gamma\xo)}\Bigr)^2 .
  \end{align*}
  Since \[\displaystyle \sum_{T<d(\xo,\gamo)\le T+5c} \e^{-\delta_\Gamma d(\xo,\gamo)}\quad \]
is uniformly bounded in $T$ as a direct consequence of  Corollary~3.8 in
\cite{MR2290453}, we have established (\ref{upbound}) with a constant $C>0$ depending only on $c$. 

It remains to prove inequality~(\ref{lowbound}). Notice first that by Lemma~\ref{joinrankone} every pair of points $(\xi,\eta)\in \Gamma(U^-\times U^+)$ can be joined by a rank one geodesic of width smaller than or equal to twice 
the width of $h$. 

We recall that by construction every orbit of the geodesic flow which enters $\overline K$ (or one of its translates by $\Gamma$) spends at least time $c$  in it. 
Using the definition of $\overline m$, Lemma~\ref{K0calL}
and the  non-negativity of the Gromov product,  we first obtain for $\gamma\in\Gamma$ with $5c\le d(\xo,\gamma\xo)\le T- 5c$
\begin{align*} & \hspace{-0.3cm}\int_0^T \d t\,  \overline m(\overline K\cap g^{-t}\gamma \overline K)\\
&\ge \int_{{\mathcal
    L}_{c,c}(\xo,\gamo)\cap\Gamma(U^-\times U^+)} \d \mu_\xo(\xi)\d \mu_\xo(\eta)\overbrace{\biggl(\int_{-c/2}^{c/2} \d s \int_0^T \d t\, {\mathbbm 1}_{\gamma \overline K} \bigl(g^{t}(\xi,\eta,s)\bigr)\biggr)}^{\ge c^2}.
\end{align*}

Recall that $r=r_0>\width(h)$ and $c>c_0+\rho\ge r+\rho$. According to Proposition~\ref{lrcinproduct} we know that for all $\gamma\in\Gamma$ with  $d(\xo,\gamma\xo)>R\,$ (with $ R>5c$ sufficiently large)
there exists an element $\beta$ in the finite set
$\Lambda\subseteq\Gamma$ with the property
$$ {\mathcal L}_{r,c}(\xo,\beta^{-1}\gamma \xo)\cap \big(U^-\times U^+\big)\supseteq (U^- \cap \rand) \times {\mathcal O}_{r,c}^-(\xo,\beta^{-1}\gamma\xo);$$
using (\ref{lrcinclus}) and $c>c_0+\rho\ge r +\rho$  we also have the inclusion
\[ {\mathcal L}_{r,c}(\xo,\beta^{-1}\gamma \xo)=\beta^{-1}{\mathcal L}_{r,c}(\beta\xo,\gamma\xo)\subseteq
\beta^{-1} {\mathcal L}_{r+\rho,c}(\xo,\gamma\xo)\subseteq \beta^{-1} {\mathcal L}_{c,c}(\xo,\gamma\xo).\]
So for all $\gamma\in\Gamma$ with $R<d(\xo,\gamma\xo)\le T-5c$ and $\beta=\beta(\gamma)\in\Lambda$  as above we have
\begin{align*}
 {\mathcal L}_{c,c}(\xo,\gamma\xo)\cap \Gamma (U^-\times U^+) & \supseteq {\mathcal L}_{c,c}(\xo,\gamma\xo)\cap \beta(U^-\times U^+)\\
 & \supseteq \beta \bigl((U^- \cap \rand) \times {\mathcal O}_{r,c}^-(\xo,\beta^{-1}\gamma\xo)\bigr)\\
 &=(\beta U^- \cap \rand) \times \beta {\mathcal O}_{r,c}^-( \xo,\beta^{-1}\gamma\xo)\end{align*}
and therefore 
\begin{align*}
\int_0^T \d t\,  \overline m(\overline K\cap g^{-t}\gamma \overline K)&\ge c^2\cdot  \int_{ (\beta U^-\cap \rand)\times \beta {\mathcal O}_{r,c}^-( \xo,\beta^{-1}\gamo)\big)} \d \mu_\xo(\xi)\d \mu_\xo(\eta)\\
&= c^2  \cdot \mu_{\xo}(\beta U^-) \mu_\xo\bigl(\beta {\mathcal O}_{r,c}^-(\xo,\beta^{-1}\gamma\xo) \bigr)\\
&\ge c^2  \cdot \mu_{\xo}(\beta U^-) \e^{-\delta_\Gamma d(\xo,\beta^{-1}\xo)}\mu_\xo( {\mathcal O}_{r,c}^-(\xo,\beta^{-1}\gamma\xo) )\\
&\ge c^2  \cdot \mu_{\xo}(\beta U^-) \e^{-\delta_\Gamma d(\xo,\beta^{-1}\xo)} \cdot \frac1{D(c)}\e^{-\delta_\Gamma d(
\xo,\beta^{-1}\gamma\xo)}\\
& \ge c^2\cdot  \min_{\beta\in\Lambda} \mu_\xo(\beta U^-) \cdot \e^{-2 \delta_\Gamma\rho} \frac1{D(c)}  \e^{-\delta_\Gamma d(
\xo,\gamma\xo)}\\
&= C'' \e^{-\delta_\Gamma d(\xo,\gamma\xo)}
\end{align*}
with a constant $C''$ depending only on $c$ and the fixed finite set $\Lambda\subseteq\Gamma$; in the last three inequalities we used the $\Gamma$-equivariance and the conformality~(\ref{conformality}) of $\mu$, the shadow lemma Proposition~\ref{shadowlemma} and the triangle inequality for the exponent.

Finally, taking the sum over all elements  $\gamma\in\Gamma$  we get
 \begin{align*}
  \int_0^T  \sum_{\gamma\in\Gamma} \overline m(\overline K\cap g^{-t}\gamma \overline K)\;\d t &\ge  \int_0^T  \sum_{\begin{smallmatrix}{\scriptscriptstyle\gamma\in\Gamma}\\{\scriptscriptstyle
      R< d(o,\gamo)\le T-5c}\end{smallmatrix}} \overline m(\overline K\cap g^{-t}\gamma \overline K)\;\d t \\
&\ge C''
\sum_{\begin{smallmatrix}{\scriptscriptstyle\gamma\in\Gamma}\\{\scriptscriptstyle
      R< d(o,\gamo)\le T-5c}\end{smallmatrix}}
\e^{-\delta_\Gamma d(\xo,\gamo)},\end{align*}
and inequality~(\ref{lowbound}) follows with the same argument as above, namely that the sums
$$\sum_{\begin{smallmatrix}{\scriptscriptstyle\gamma\in\Gamma}\\{\scriptscriptstyle
       d(o,\gamo)\le R}\end{smallmatrix}}
\e^{-\delta_\Gamma d(\xo,\gamo)} \qquad\text{and}\quad \sum_{\begin{smallmatrix}{\scriptscriptstyle\gamma\in\Gamma}\\{\scriptscriptstyle
      T-5c< d(o,\gamo)\le T}\end{smallmatrix}}
\e^{-\delta_\Gamma d(\xo,\gamo)}$$
are uniformly bounded in $T$.\end{proof}

\section{Conclusion and a construction of convergent groups}\label{conclusion}

We now summarize all the previously collected results in the weakest possible setting:
\begin{theorem}\label{HTSweak}\
Let $\XX$ be a proper Hadamard space and $\Gamma<\is(\XX)$ a discrete rank one group. 
For $\delta>0$ let $\mu$ be a $\delta$-dimensional   $\Gamma$-invariant  conformal density
normalized such that $\mu_\xo(\rand)=1$, and 
 $\overline m_\Gamma$ the weak Ricks' measure on $\quotient{\Gamma}{ [\SX]}$ associated to the quasi-product geodesic current $\overline\mu$ defined by (\ref{overlinemudef}). Then exactly one of the following two complementary cases holds, and the statements (i) to (iii) 
 are equivalent in each case: \\[2mm]
  \noindent 
1.~Case:
\begin{enumerate}
\item[(i)] 
$\sum_{\gamma\in\Gamma} \e^{-\delta d(\xo,\gamma\xo)}$ diverges.
\item[(ii)] $\mu_\xo(\radlim)=1$.
\item[(iii)] $(\quotient{\Gamma}{ [\SX]}, g_\Gamma, 
\overline m_\Gamma)$ is  conservative.
\end{enumerate}
2.~Case:
\begin{enumerate}
\item[(i)] 
$\sum_{\gamma\in\Gamma} \e^{-\delta d(\xo,\gamma\xo)}$ converges.
\item[(ii)] $\mu_\xo(\radlim)=0$.
\item[(iii)] $(\quotient{\Gamma}{ [\SX]}, g_\Gamma, 
\overline m_\Gamma)$ is  dissipative.
\end{enumerate}
\end{theorem}

We remark that the first case can only happen if $\Gamma$ is divergent and if $\delta=\delta_\Gamma$. In this case there are several well-known  additional statements: The $\delta_\Gamma$-dimensional    $\Gamma$-invariant conformal density $\mu$ is unique up to multiplication by a scalar.  Moreover it follows from Lemma~\ref{ergodicity} that $\mu$ is quasi-ergodic in the sence that every $\Gamma$-invariant Borel subset $A\subseteq \rand$ either has zero or full measure with respect to any measure $\mu_x$ in $\mu$. According to Proposition~\ref{atomicpart}, $\mu$ is also non-atomic. 

Obviously, if $\delta>\delta_\Gamma$, then we are always in the second case. 
Moreover, in the second case the measure $\overline m_\Gamma$ is infinite and we also have non-ergodicity of the dynamical system $(\quotient{\Gamma}{ [\SX]}, g_\Gamma, \overline m_\Gamma)$ unless the measure $\overline m_\Gamma$ is supported on a single divergent orbit $\{g_\Gamma^t [v]\colon t\in\RR\}$ for some $v\in\quotient{\Gamma }{\SX}$; this follows directly from the paragraph before Theorem~\ref{Hopfindividual}. 

Since for $\delta>\delta_\Gamma$ we are always in the dissipative case we will formulate the subsequent results only for $\delta=\delta_\Gamma$. 
Under the presence of a zero width rank one geodesic with extremities in the limit set we get the following statement which implies Theorem~B from the introduction:
\begin{theorem}\label{HTS}\
Suppose $\Gamma<\is(\XX)$ is a discrete rank one group  with  the extremities of a {\hl zero width} rank one geodesic  in its limit set.  Let $\mu$ be a $\delta_\Gamma$-dimensional   $\Gamma$-invariant  conformal density
normalized such that $\mu_\xo(\rand)=1$,
and $m_\Gamma$ the associated Ricks' measure on $\quotient{\Gamma}{ \SX}$. Then exactly one of the following two complementary cases holds, and the statements (i) to (iv) are equivalent in each
case: \\[2mm]
1.~Case:
\begin{enumerate}
\item[(i)] $\sum_{\gamma\in\Gamma} \e^{-\delta_\Gamma d(\xo,\gamma\xo)}$ diverges.
\item[(ii)] $\mu_\xo(\radlim)=1$.
\item[(iii)] $(\quotient{\Gamma}{ \SX}, g_\Gamma, 
m_\Gamma)$ is  conservative.
\item[(iv)] $(\quotient{\Gamma}{ \SX}, g_\Gamma, 
m_\Gamma)$ is  ergodic  and $m_\Gamma$ is not supported on a single divergent orbit. 
\end{enumerate}

\noindent 2.~Case:
\begin{enumerate}
\item[(i)] $\sum_{\gamma\in\Gamma} \e^{-\delta_\Gamma d(\xo,\gamma\xo)}$ converges.
\item[(ii)] $\mu_\xo(\radlim)=0$.
\item[(iii)] $(\quotient{\Gamma}{ \SX}, g_\Gamma, 
m_\Gamma)$ is  dissipative.
\item[(iv)] $(\quotient{\Gamma}{ \SX}, g_\Gamma, 
m_\Gamma)$ is  non-ergodic 
unless $m_\Gamma$ is supported on a single divergent orbit. 
\end{enumerate}
\end{theorem}

\noindent Let us discuss the relation between Theorem~\ref{HTS} above and Theorem~\ref{HTSweak} in the case that   $\Lim$ contains the extremities of a zero width rank one geodesic and $\delta=\delta_\Gamma$: If $\Gamma$ is divergent, then according to Theorem~\ref{conservativestatement}  the weak Ricks' measure is equal to the Ricks' measure. So the statements in the first case of Theorem~\ref{HTSweak} are only supplemented by the fact that the dynamical systems are ergodic. 

For a convergent group $\Gamma$ it is well-known that there can exist many different $\delta_\Gamma$-dimensional  $\Gamma$-invariant  conformal densities. So first of all it is possible to obtain several distinct weak Ricks' measures $\overline m_\Gamma$ associated to different conformal densities. And even if the same $\delta_\Gamma$-dimensional  $\Gamma$-invariant  conformal density is used in the construction, the Ricks' 
measure $m_\Gamma$ can be different from the {\hl weak} Ricks' measure $\overline m_\Gamma$ (as it is supported on an a priori smaller set). The statements in Theorem~\ref{HTS} above and Theorem~\ref{HTSweak} for the second case therefore apply to any (weak) Ricks' measure constructed from a $\delta_\Gamma$-dimensional  $\Gamma$-invariant conformal density.

In order to obtain Theorem~C from the introduction, we have  to relate our new results to the Main Theorem in \cite{LinkPicaud}. Since  the measure $\overline \mu$ on $\ndpt\SX$  is used in 
Knieper's construction, Knieper's measure coincides with Ricks' measure on the set $
\quotient{\Gamma}{\zero}$. As in the divergent case the  support of both Knieper's and Ricks' measure is $\quotient{\Gamma}{\zero}$,  the divergent case of the Main Theorem in \cite{LinkPicaud} remains true under the weaker hypothesis that $\Gamma$ is a discrete rank one group. By Lemma~\ref{consdiss} we further get that the equivalent conditions in  the convergent case hold under the same weaker condition. So the existence of a periodic geodesic without parallel perpendicular  Jacobi field  in $\quotient{\Gamma}{\XX}$ is not a necessary hypothesis in the Main Theorem of \cite{LinkPicaud} and  we immediately get Theorem~C from the introduction.

Finally 
I want to mention that for finite $m_\Gamma$ -- the case treated in the article \cite{Ricks} by R.~Ricks -- we are always in the first case; this follows easily from the fact that finite measure spaces are conservative.  Ricks further showed (\cite[Theorem~4]{Ricks}) that if $\XX$ is geodesically complete, $m_\Gamma$ is finite and $\Lim=\rand$, then $(\quotient{\Gamma}{ \SX}, g_\Gamma, m_\Gamma)$ is mixing unless $\XX$ is isometric to a tree with all edge lengths in $c\ZZ$ for some $c>0$. 

To conclude this article I want to describe a construction of convergent rank one groups whose idea goes back to F.~Dal'bo, J.P.~Otal and M.~Peign{\'e} (\cite{MR1776078}, see also \cite{MR3220550}). We first give a criterion for the critical exponent of a divergent subgroup of a rank one group which extends Theorem~3.2 in \cite{MR3220550}:
\begin{proposition} Let $\XX$ be a proper Hadamard space and $\Gamma<\is(\XX)$ a discrete rank one group. If $H<\Gamma$ is a divergent subgroup with $L_H\subsetneq \Lim$, then its critical exponent satisfies $\delta_H<\delta_\Gamma$.
\end{proposition}
\begin{proof} As $L_H\subsetneq \Lim$ we may choose a point $\xi\in \Lim\setminus L_H$. Since $L_H$ is a closed subset of $\rand$ there exists an open neighborhood $U\subseteq\rand$ of $\xi$ such that $U\cap L_H=\emptyset$. As $\Gamma$ is a discrete rank one group, Theorem~2.8 in \cite{MR656659}  implies the existence of a rank one element $g\in\Gamma$ \st $g^+\in U$. Let $V^-\subseteq\rand$, $V^+\subseteq U$ be small neighborhoods of $g^-$, $g^+$ respectively. Taking a rank one element $ \gamma\in \Gamma$ independent from $g$ and making $V^-$ smaller if necessary we have $\{ \gamma^-, \gamma^+\}\cap V^-=\emptyset$. Using the north-south dynamics Lemma~\ref{dynrankone} (b) we know that for $N\in\NN$ sufficiently large the rank one element
\[ \widetilde \gamma=g^N \gamma g^{-N}\in\Gamma\]
has both fixed points in $V^+\subseteq U$. Replacing $\widetilde\gamma$ by $ \widetilde\gamma^M$ for some $M\in\NN$ large enough we may further assume that
\[  \widetilde\gamma (\rand\setminus U)\subseteq U \quad\text{and }\ \widetilde \gamma^{-1} (\rand\setminus U)\subseteq U.\]
We now consider the free product
$G=H\ast  \widetilde\gamma<\Gamma$; the set
\[ \{h_1\widetilde\gamma h_2\widetilde\gamma \cdots h_{k-1}\widetilde\gamma h_k\widetilde \gamma\colon k\in\NN, h_i\in H\setminus\{e\}\}\]
is obviously a subset of $G$ and hence of $\Gamma$. For any $s>0$  the Poincar{\'e} series $P_\Gamma(s)$ of $\Gamma$ 
then satisfies 
\begin{align*}
P_\Gamma(s)&= \sum_{\gamma\in\Gamma}\e^{-sd(\xo,\gamma\xo)}\ge \sum_{k=1}^\infty 
\ \sum_{h_1,\ldots h_k\in H\setminus\{e\}} \e^{-sd(\xo, h_1 \widetilde\gamma h_2\cdots \widetilde\gamma h_k\widetilde \gamma \xo)}\\
&\ge \sum_{k=1}^\infty \e^{-s k d(\xo,\widetilde \gamma\xo)}
\sum_{h_1,\ldots h_k\in H\setminus\{e\}} \e^{-sd(\xo, h_1\xo)}\e^{-sd(\xo,h_2\xo)}\cdots \e^{- s d(\xo, h_k\xo)}\\
&= \sum_{k=1}^\infty \left(\e^{-s  d(\xo,\widetilde \gamma\xo)}\right)^k\cdot 
\left( \sum_{h\in H\setminus\{e\}} \e^{-sd(\xo, h\xo)}\right)^k. \end{align*}

Since $H$ is divergent, the sum $\displaystyle \sum_{h\in H\setminus\{e\}}\e^{-sd(\xo,h\xo)}$ tends to infinity as $s\searrow \delta_H$. Hence there exists $s_0>\delta_H$ \st 
\[ \e^{-s_0  d(\xo,\widetilde \gamma\xo)} \cdot  \sum_{h\in H\setminus\{e\}} \e^{-s_0d(\xo, h\xo)}>1;\]
for this parameter $s_0$ the Poincar{\'e} series $P_\Gamma(s_0)$ diverges, hence 
$\delta_H<s_0\le \delta_\Gamma$.
\end{proof}
Notice that $H$ need not be a rank one group. However, as in \cite{MR3220550} the above proposition allows to produce plenty of convergent discrete rank one isometry groups of any Hadamard space admitting a rank one isometry. The only novelty in the proof compared to the one given by M.~Peign{\'e}  in \cite{MR3220550} is the fact that the convergent subgroup is rank one (and hence is an example for a group in which the second case of Hopf-Tsuji-Sullivan dichotomy holds). 
\begin{corollary}
Let $\XX$ be a proper Hadamard space \st  $\is(\XX)$ 
contains two  independent rank one elements 
$h,g$. 
Then there exist $N,M\in\NN$ \st  the subgroup $G$ of $\is(\XX)$ generated by  
\[ \{ g^{-nN} h^M g^{nN}\colon n\in \NN_0\}\]
is a convergent discrete rank one group.
\end{corollary}
\begin{proof}
Let $U^-,U^+, V^-,V^+\subseteq\ganz$ be pairwise disjoint neighborhoods of $h^-,h^+, g^-,g^+$. Thanks to Lemma~\ref{dynrankone} (b) there exist $M,N\in\NN$ \st
\begin{equation}\label{inclusionofLH} h^{\pm M}(V^-\cup V^+)\subseteq U^\pm\quad\text{and }\  g^{\pm N}(U^-\cup U^+)\subseteq V^\pm. \end{equation}
This implies that $G$ acts freely on $\XX$  and hence that $G$ is discrete;
moreover, the limit set 
 $L_{G }$ of $G $  contains the set
\[\{ g^{-nN} h^-, g^{-nN} h^+\colon  n\in \NN_0\},\]
so $L_{G}$ is infinite. 
Hence according to Lemma~\ref{inflimset} $G $ is a rank one group. 
The limit set $L_H$ of the conjugate discrete subgroup $H=g^{-N} G g^N<\is(\XX)$  is contained in $L_G$ and also in $V^-$ by (\ref{inclusionofLH}). Since $ h^+\in L_G$, $h^+\notin V^-$ we get $L_H\subsetneq L_{G}$. Obviously we also have $\delta_H=\delta_G$, hence the proposition above implies that $H$ must be convergent. As conjugate groups  are simultanously convergent or divergent we conclude that $G$  is  convergent. 
\end{proof}
Notice that the isometry group of a Hadamard space $\XX$ contains two independent rank one elements whenever it admits a discrete rank one subgroup. So the above construction in particular allows to construct plenty of convergent rank one subgroups in a given rank one discrete isometry group of $\XX$. 

\section*{Acknowledgements}
The author would like to thank Russel Ricks for answering her questions concerning his article \cite{Ricks}  and for pointing out a mistake in a previous version of this article. She also thanks Marc Peign{\'e} for his comments on the preprint. Finally she would like to thank the referee for carefully reading the article and for pointing out a gap in the original version of the proof to Proposition~\ref{largewidthgiveszerowidth}.

\def\cprime{$'$}
\providecommand{\href}[2]{#2}
\providecommand{\arxiv}[1]{\href{http://arxiv.org/abs/#1}{arXiv:#1}}
\providecommand{\url}[1]{\texttt{#1}}
\providecommand{\urlprefix}{URL }

\end{document}